\documentclass[a4paper,12pt]{amsart}
\makeatletter

\newcommand{\Rmnum}[1]{\expandafter\@slowromancap\romannumeral #1@}
\makeatother
\usepackage{amsthm}
\usepackage{amsmath}
\usepackage{amssymb}
\usepackage[ansinew]{inputenc}
\usepackage{amscd}
\usepackage[ansinew]{inputenc}
\usepackage[mathscr]{euscript}
\usepackage{color}
\usepackage[all]{xy}
\usepackage{graphicx}

\newtheorem{teorema}{Theorem}[section]
\newtheorem{lema}[teorema]{Lemma}
\newtheorem{corolario}[teorema]{Corollary}
\newtheorem{prop}[teorema]{Proposition}
\newtheorem{af}[teorema]{\textbf{Claim}}

\newtheorem{as}[teorema]{\textbf{Assumption}}
\newtheorem{definicao}[teorema]{\textbf{Definition}}
\newtheorem{observacao}[teorema]{\textbf{Remark}}
\newcommand{\PC}[1]{\ensuremath{\left(#1\right)}}

\providecommand{\norm}[1]{\lVert#1\rVert}

\newcommand{\N}{\mathbb{N}} %conjunto dos naturais
\newcommand{\R}{\mathbb{R}} %conjunto dos reais
\def\R{\mathbb R}

\def\N{\mathbb N}

\def\geq{\geqslant}
\def\leq{\leqslant}

\def\*{\color{red}\blacksquare}

\date{\today} % delete this line to display the current date
\thanks{ This work is based on the Ph.D. Thesis of the second author.   J. V. was partially supported by FAPESP 2009/17153-9.
D.S. was partially supported by  CNPq 305537/2012-1.}
\begin{document}
\author{Daniel Smania and Jos\'{e} Vidarte}
\title[Existence of $C^{k}$-Invariant Foliations for Lorenz-Type Maps]
{Existence of  $C^{k}$-Invariant Foliations for Lorenz-Type Maps}
{\small\address{Departamento de Matem\'atica,
ICMC-USP, Caixa Postal 668,  S\~ao Carlos-SP,CEP 13560-970
S\~ao Carlos-SP, Brazil}}
{\small\email{smania@icmc.usp.br}}
{\small\urladdr{www.icmc.usp.br/$\sim$smania/}}
{\small\address{Universidade Federal de Itajub\'{a}, Instituto de Matem\'{a}tica e Computa\c{c}\~{a}o.
Avenida BPS, 1303 Pinheirinho 37500903 - Itajub\'{a}, MG - Brasil}}
\email{vidarte@unifei.edu.br}
%\email{vidarte@icmc.usp.br}

%\thanks{Work supported by DGICYT Grant MTM2009-08933}
%
%\keywords{Polynomial maps, Newton polyedron at infinity, mixed volumes}

%%%%%%%%%%%%%%%%%%%%%%%%%%%%%%%%%%%%%%%%%%%%%%%%%%%%%%%%%%%%%%%%%%%%Resum
%\noindent\hrulefill
%\setcounter{tocdepth}{1}
\begin{abstract}
%In this paper we give some conditions over Lorenz type map $T$ and show that there exists a $C^{k}$ foliation which is invariant under $T$.
%We prove the existence of a $C^{k}$-invariant foliation for Lorenz type map under certain properties.
%In this paper we give sufficient conditions over Lorenz type map $T$ to show
In this paper under similar conditions to that Shaskov and Shil'nikov  [1994] we show that a $C^{k+1}$ Lorenz-type map $T$ has a $C^{k}$ foliation which is invariant under $T$. This allows us to associate $T$ to a $C^{k}$ one-dimensional transformation.
%In this paper we give sufficient conditions for the existence of a $C^{k}$-invariant foliation for Lorenz type map.
%In this paper we show that if a Lorenz type map $T$ satisfies certain properties, then there exists a $C^{k}$ foliation which is invariant under $T$.
\end{abstract}
\maketitle
%%%%%%%%%%%%%%%%%%%%%%%%%%%%%%%%%%%%%%%%%%%%%%%%%%%%%%%%%%%%%%%%%%%%%%%%%%%%%%%%%%%%%%%%%%%%%%%%%%%%%%%%%%%%%%%%%%

%\tableofcontents

%%%%%%%%%%%%%%%%%%%%%%%%%%%%%%%%%%%%%%%%%%%%%%%%%%%%%%%%%%%%%%%%%%%%%%%%%%%%%%%%%%%%%%%%%%%%%%%%%%%%%%%%%%%%%%%%%%
\setcounter{tocdepth}{1}
\tableofcontents
\section{ Introduction }\label{S:1}

The\textbf{ geometric Lorenz model} is an important example in dynamical systems, which was initially studied by Guckenheimer and Williams \cite{Guckenheimer1, Guckenheimer2, Williams} and Afraimovich, Bykov and Shil'nikov \cite{Shilnikov}. Their aim was to construct a simple mechanism which can give similar results to that Lorenz system
\begin{equation*}
\displaystyle (\dot{x},\dot{y},\dot{z})=(10(y-x),28 x-y-xz, -\frac{8}{3} z+xy),
 \end{equation*}
introduced by Edward Lorenz \cite{Lorenz}. In this system, Edward Lorenz numerically found that most solutions tended to a
certain attracting set, so-called  \emph{Lorenz Attractor} or  \emph{``strange" attractor},  and in so doing, he produced an important early example of `` chaos ". Another fact that was noted by Lorenz: the \emph{Lorenz Attractor} has  \emph{ sensitivity to initial conditions} (\emph{the butterfly effect}). No matter how close two solutions start, they will have a quite different  behaviour in the future. The \textbf{geometric Lorenz model} has been analysed topologically and proved to possess a ``strange" attractor with sensitive dependence on initial conditions. From these facts, we know that the \textbf{geometric Lorenz model} is crucial in the study of dynamical systems.  For more details, see Viana \cite{Viana}.

Given a $C^{k+1}$ geometric Lorenz flow $X$ on $\R^{n+2},$  by definition there exists a $C^{k+1}$ Poincar\'{e} map $T_{_{X}}:D^{*}\to D$ , often so-called Lorenz-type map  \cite{Shilnikov}. In Shaskov and Shil'nikov\cite{Shashkov} the authors showed that if a  $C^{2}$  Lorenz-type map $T_{_{X}}$ satisfies certain conditions, then there exists a $C^{1}$ foliation which is invariant under $T_{_{X}}$. It allows us to associate $T_{_{X}}$ to a $C^{1}$ one dimensional Lorenz like map $f_{_{X}}:[a,b]\setminus \{c\} \to [a,b]$ .  This association is so-called the reduction transformation $\mathcal{R},$ so we have $\mathcal{R}T_{_{X}}=f_{_{X}}$. This result allows $T_{_X}$ to be described in terms of a $C^{1}$ one-dimensional map.

Since the most deep results in  one-dimensional dynamics (as the phase-parameter relations in Jakobson's Theorem \cite{jakobson} and renormalization theory) relays on the study of sufficiently smooth families of transformations, to transfer this result to geometric Lorenz flow  we need to study the smoothness of the reduction transformation $\mathcal{R}$. There are already impressive results using this approach ( see Rychlik \cite{Rychlik}, Rovella \cite{Rovella}, Morales, Pacifico and Pujals \cite{Pacifico}, , Ara\'{u}jo and Varandas \cite{AraVar} ), however if we had a more deep knowledge of the regularity of $\mathcal{R}$ then  far more significative results could be achieved.

In this work we extend the main result of Shil'nikov and Shaskov \cite[Theorem]{Shashkov} as well as  \cite[Corollary 4.2]{Rychlik} of Rychlik, \cite[Proposition, p. 241]{Rovella} of Rovella and \cite[Lema 2]{Pacifico} of Morales, Pacifico and Pujals  . That is, we show that  if a $C^{k+1}$ Lorenz type map $T$  satisfies certain conditions (see Assumption \ref{Assumptions}), then there exists a $C^{k}$ foliation which is invariant under $T_{_{X}}$. This theorem allows us to introduce new coordinates $\{(x,\eta)\}$ in $D$  such that the map $T_{X}$ has the form $\overline{T}_{_{X}}(x,\eta)=(\overline{F}_{_{X}}(x,\eta),\overline{G}_{_{X}}(\eta))$ (see Afraimovich and Pesin \cite[P. 178]{pesin}); where $\overline{F}_{_{X}}$ and $\overline{G}_{_{X}}$ are $C^{k}$ functions, so $T_{_{X}}$ can be associate to a $C^{k}$ one-dimensional transformation $\overline{G}_{_{X}}:[a,b]\setminus \{c\} \to [a,b]$. This association would allow us to study the dynamical properties of the original flow using powerful techniques of  $C^{k}$ one-dimensional dynamics. Moreover, the result of this work can be useful in studying maps conside\-red in Robinson \cite{Robinson1}, Rychlik \cite{Rychlik}, Rovella \cite{Rovella},  Morales, Pacifico and Pujals \cite{Pacifico}, Ara\'{u}jo and Varandas \cite{AraVar}, Araujo and Pacifico \cite{arapa} and in some other cases.
\begin{figure}[h]
\centering\includegraphics[height=9cm,width=9cm]{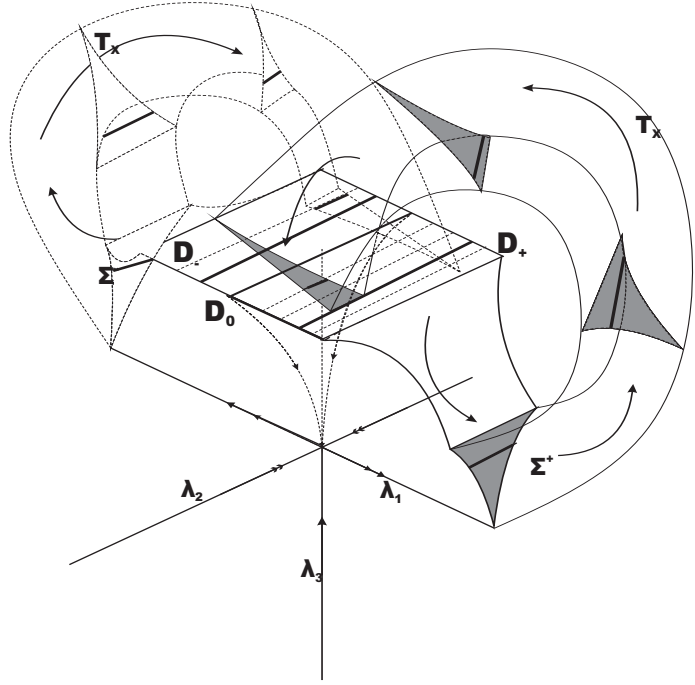}
 \caption{Geometric Lorenz flow in $\R^{3}.$}
 \label{IF}
\end{figure}

\section{ Statement of the Main Result}\label{S:1}

Let $\R^{n+1}:=\R^{n}\times\R$ be a $(n+1)$-Euclidean space . From now on, the symbol $\parallel . \parallel$ denotes a norm in $\R^{n},$ if applied to a vector or for the corresponding matrix norm if applied to a matrix. We also use the notation $$\norm{.}_{D}=\displaystyle \sup_{(x,y)\in D^{*}}\norm{.}$$  for norms of matrices and vector functions on $D^{*}$.

Define
\begin{eqnarray}\label{D}
D_{\empty}&:=&\{(x,y)\in \R^{n+1}:\norm{x}\leq 1, |y|\leq 1\},\nonumber \\
D_{+}&:=&\{(x,y)\in D: y>0\},\nonumber \\
D_{-}&:=&\{(x,y)\in D:y<0\},\nonumber \\
D_{0}&:=&\{(x,y) \in D:y=0\},\nonumber\\
D^{*}&:=&D_{-}\cup D_{+}=D\setminus D_{0}.
\end{eqnarray}
Notice that the sets $D_{+}$ and $D_{-}$ are separate by the hyperplane $D_{0}.$

Let us consider the map $T: D^{*} \to D $ given by
\begin{equation}\label{T0}
T(x,y)=(F(x,y),G(x,y))=(\overline{x},\overline{y}),
\end{equation}
where the vector function $F$ and the scalar function $G$ are differentiables on $D^{*}$ and $\partial_{y}G(x,y)$ is non-vanishing on $D^{*}.$

\begin{definicao}\label{A,B,C} We define the following functions:
\begin{equation*}\label{A}
A(x,y):=\partial_{x}F(x,y)(\partial_{y}G(x,y))^{-1},
\end{equation*}
\begin{equation*}\label{B}
B(x,y):=\partial_{y}F(x,y)(\partial_{y}G(x,y))^{-1},
\end{equation*}
\begin{equation*}\label{C}
C(x,y):=\partial_{x}G(x,y)(\partial_{y}G(x,y))^{-1}.
\end{equation*}
\end{definicao}
Here $A(x,y)$ is a $n \times n$ matrix, $B(x,y)$ is a $n$-column vector and $C(x,y)$ is a $n$-row vector.

\begin{as}\label{Assumptions}
We assume the following conditions hold on $T$:
\begin{itemize}
\item[$(L_{1})$] The functions $F$ and $G$  have the forms
\begin{displaymath}
F(x,y) = \left\{ \begin{array}{ll}
x_{+}^{\ast}+|y|^{\alpha}[B_{+}^{\ast}+\varphi_{+}(x,y)], &  y>0,\\
x_{-}^{\ast}+|y|^{\alpha}[B_{-}^{\ast}+\varphi_{-}(x,y)], & y<0,
\end{array} \right.
\end{displaymath}
\begin{displaymath}
G(x,y) = \left\{ \begin{array}{ll}
y_{+}^{\ast}+|y|^{\alpha}[A_{+}^{\ast}+\psi_{+}(x,y)], & y>0,\\
y_{-}^{\ast}+|y|^{\alpha}[A_{-}^{\ast}+\psi_{-}(x,y)], & y<0,
\end{array} \right.
\end{displaymath}
in a neighborhood of $D_{0},$ where $A_{+}^{\ast}, A_{-}^{\ast}$ are nonzero constants; $\alpha$ represents a strictly positive constant and the functions $\varphi_{\pm},$ $\psi_{\pm}$ are of class $C^{k+1}$. The derivatives of $\varphi_{\pm}$ and $\psi_{\pm}$ are uniformly bounded with respect to $x$ and satisfy the  estimates:
\begin{equation*}\label{L1}
\displaystyle \left\|\frac{\partial^{l+m}\varphi_{\pm}(x,y)}{\partial x^{l} \partial y^{m}} \right\|\leq K |y|^{\gamma-m}, \quad \displaystyle \left\|\frac{\partial^{l+m}\psi_{\pm}(x,y)}{\partial x^{l} \partial y^{m}} \right\|\leq K |y|^{\gamma-m},
\end{equation*}
where $\gamma>k-1,$ $K$ is a positive constant, $l=0,1,\ldots,k+1$ $m=0,1,\ldots,k+1$ and $l+m\leq k+1$.
\item[$(L_{2})$] The following inequality holds: \begin{equation*}\label{L2}
1-\|A\|_{D}>2\sqrt{\|B\|_{D}\|C\|_{D}}.
\end{equation*}
\item[$(L_{3})$]The following relations hold:

\item[(a)]
\begin{equation*}\label{L3}
\frac{(2!)^{2}\left(||A||_{D}+||C||_{D}||B||_{D}\right)\displaystyle\max_{m+n=1}\{(||A||_{D}+||B||_{D})^{m}(||C||_{D}+1)^{n}\}}{(||\partial_{y}G||_{D})^{-1}\left(1+||A||_{D}+\sqrt{(1-||A||_{D})^{2}-4||B||_{D}||C||_{D}}\right)^{2}}<1.
\end{equation*}
\item[(b)] for $k\geq 2 $
  \begin{equation*}\label{L30}
\frac{(2k!)^{2}\left(||A||_{D}+||C||_{D}||B||_{D}\right)\displaystyle\max_{m+n=k}\{(||A||_{D}+||B||_{D})^{m}(||C||_{D}+1)^{n}\}}{(||\partial_{y}G||_{D})^{-k}\left(1+||A||_{D}+\sqrt{(1-||A||_{D})^{2}-4||B||_{D}||C||_{D}}\right)^{2}}<1,
\end{equation*}
and
\begin{equation*}\label{L301}
 \norm{\partial_{y}G}_{D}\geq \frac{1}{4}\quad  \mbox{or} \quad \norm{\partial_{x}F}_{D}\geq \frac{1}{4}.
\end{equation*}
\end{itemize}
\end{as}
The following  set will be useful for defining the domains of several maps:
\begin{equation*}\label{Dx}
D_{x}:=\{x\in \R^{n} \quad \mbox{for which there exists a  } y \in \R \quad \mbox{with} \quad (x,y) \in D\}.
\end{equation*}
Given a map $h:U\subset\R^{n}\to \R,$ we defined its graph as
\begin{equation*}
\mbox{graph}(h):=\{(x,h(x)):x\in U\}.
\end{equation*}
\begin{definicao}\label{F1}
A family of functions $\mathcal{F}_{D}=\{h(x)\}$ is called a foliation of $D$ with $C^{m}$ leaves $(m\geq 0)$ given by the graphs of functions $y=h(x)$ if the following three conditions are satisfied:
\begin{itemize}
\item[$(a)$] The domain $Dom(h(x))$ of  every function $h(x) \in \mathcal{F}_{D}$ is an open and connected set in $D_{x}$ and its graph lies entirely in $D$;
\item[$(b)$] for every point $(x_{0},y_{0})\in D$ there is a unique function $h(x)\in \mathcal{F}_{D}$ such that $x_{0}\in Dom(h(x))$ and $y_{0}=h(x_{0}),$ this function will be denoted by $h(x;x_{0},y_{0})$;
\item[$(c)$] for every  point $(x_{0}, y_{0})\in D$ the function $x \mapsto h(x;x_{0},y_{0})$  is of class $C^{m}.$
\end{itemize}
\end{definicao}
The graphs of the functions $h(x)$ are called the leaves of $\mathcal{F}_{D}$ and the leaf that contain $(x_{0},y_{0})\in D$ will be denoted by $\mathcal{F}_{(x_{0},y_{0})}$.

\begin{definicao}\label{DfCr}
A foliation $\mathcal{F}_{D}$ is called $C^{r}$-foliation $(r\geq0)$ if the function $$(x; x_{0},y_{0}) \mapsto h(x; x_{0},y_{0})$$ is of class $C^{r}.$
\end{definicao}
\begin{definicao}\label{ITF}
A foliation $\mathcal{F}_{D}$ is called $T$-invariant if
\begin{itemize}
\item[(a)] the hyperplane $D_{0} \in \mathcal{F}_{D}$;
\item[(b)] for each leaf $\mathcal{F}_{(x_{0},y_{0})}\in \mathcal{F}_{D}$, with $\mathcal{F}_{(x_{0},y_{0})}\neq D_{0}$, there is $\mathcal{F}_{T(x_{0},y_{0})} \in \mathcal{F}$ such that $T(\mathcal{F}_{(x_{0},y_{0})})\subset \mathcal{F}_{T(x_{0},y_{0})}.$
\end{itemize}
\end{definicao}
\begin{figure}[h]
\centering
\includegraphics[height=5cm,width=9cm]{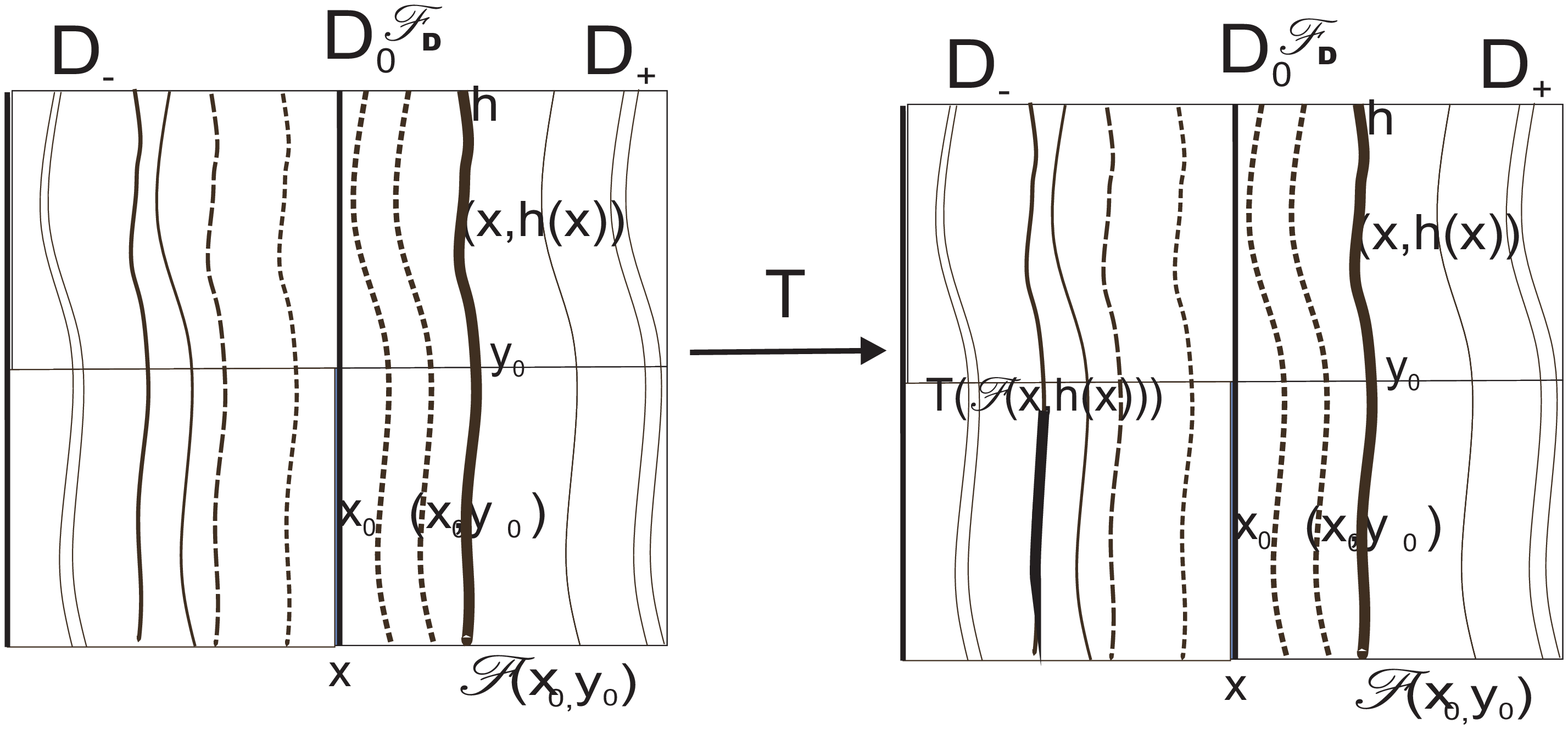}
 \caption{Geometric interpretation of a $T$-invariant foliation.}
 \label{IF}
\end{figure}
\begin{observacao}\label{R:1}
Suppose that  $\overline{\nu}:D\subset \R^{n+1} \to \R^{n}$ is a $C^{k}$ function completely integrable, that is, there exists a solution for the initial value problem for the differential equation
\begin{equation}\label{FI}
 \nabla y(x)=\overline{\nu}(x,y(x)), y(x_{0})=y_{0},
\end{equation}
for all $(x_{0},y_{0})\in D,$
where $y:U(x_{0})\subset D \to [-1,1]$ and $U(x_{0})$ is a neighborhood of $x_{0}.$
Then, by using Frobenious-Dieudonn\'{e} Theorem \cite[Theorem 10.9.5]{Dieudonne} we have that
 \begin{equation*}
 \mathcal{\overline{F}}_{D}:=\{\mbox{graph}(y): \quad h \quad \mbox{satisfing \eqref{FI}}\},
  \end{equation*}
  determines a foliation, that is,  the leaves are the graphs of the solutions of the differential equation defined by the function $\overline{\nu}: D \to \R^{n}.$
\end{observacao}

We are ready to state our main result.
\begin{teorema}[\textbf{Main Theorem}]\label{MT}
Suppose that the map $T$ satisfies Assumption \ref{Assumptions}. Then, there is  a $T$-invariant $C^{k}$-foliation $\mathcal{F}_{D}$ with $C^{k+1}$ leaves.
\end{teorema}
As a byproduct of the preceding theorem we also have the following useful corollary, that say us that if the map $T$ satisfies Assumption \ref{Assumptions} it can be introduced new coordinates $\{(x,\eta)\}$ in $D$  such that the map $T$ has the form of skew-product $\overline{T}(x,\eta)=(\overline{F}(x,\eta),\overline{G}(\eta));$ where $\overline{F}$ and $\overline{G}$ are $C^{k}$ functions, so $T$ can be associate to a one-dimensional transformation $\overline{G}:[a,b]\setminus \{c\} \to [a,b]$ of class $C^{k}.$
\begin{corolario}\label{CMT}
Suppose that the map $T$ satisfies Assumption \ref{Assumptions}. Then, there exists a change of variable $\chi:D \to D $ such that $T$ can be associate with a skew-product  $\overline{T} :D^{*}\to D $ of class $C^{k}$ such that the  diagram 
\begin{equation*}
\xymatrix{
D^{*} \ar[d]^\chi \ar[r]^T & D \ar[d]^\chi \\
D^{*} \ar[r]^{\overline{T}} & D }
\end{equation*}
 is commutative, that is, $\chi \circ T= \overline{T} \circ \chi$ on $D^{*}.$
\end{corolario}
\begin{proof}
The details can be found in \cite[P. 178]{pesin}.
\end{proof}
\section{Overview of the Proof of Main Theorem \ref{MT} }
The principal aim of this section is to sketch the proof of our main theorem.
%In this section we outline  the ideas to prove Theorem \ref{MT}.
%\newcounter{quest0}
%\begin{list}{\textbf{Idea \arabic{quest0}:}}{\usecounter{quest0}
%\setlength{\labelwidth}{-2mm} \setlength{\parsep}{0mm}
%\setlength{\topsep}{0mm} \setlength{\leftmargin}{0mm}}
%\renewcommand{\labelenumi}{(\alph{enumi})}
%\item Defining foliations through functions completely integrable.
%
%Suppose that  $\overline{\nu}:D\subset \R^{n+1} \to \R^{n}$ is a $C^{k}$ function completely integrable (C.I.), that is, there exists a solution for the initial value problem for the differential equation
%\begin{equation}\label{FI}
% \nabla y(x)=\overline{\nu}(x,y(x)), y(x_{0})=y_{0},
%\end{equation}
%for all $(x_{0},y_{0})\in D,$
%where $y:U(x_{0})\subset D \to [-1,1]$ and $U(x_{0})$ is a neighborhood of $x_{0}.$
%Then, by using Frobenious-Dieudonn\'{e} Theorem \cite[Theorem 10.9.5]{Dieudonne} we have that
% \begin{equation*}
% \mathcal{\overline{F}}_{D}:=\{\mbox{graph}(h): \quad h \quad \mbox{is solution of Eq. \eqref{FI}}\},
%  \end{equation*}
%  determines a foliation, that is,  the leaves are the graphs of the solutions of the differential equation defined by the function $\overline{\nu}: D \to \R^{n}.$
%%  \begin{figure}[h]
%%\centering
%%\includegraphics[height=9cm,width=7cm]{Foliation4.pdf}
%% \caption{Idea 1}
%% \label{IF}
%%\end{figure}
%
%\item How to build a $T$-invariant foliation?

\subsection{The big picture} Bearing in mind the Remark \ref{R:1} and following the ideas of Robinson \cite{Robinson}. The foliation $\mathcal{F}_{D}$ of Theorem \ref{MT} will be obtained as the integral surfaces of a $C^{k}$ completely integrable function $\nu:D\subset \R^{n+1} \to \R^{n},$ which will be a fixed point of an appropriate graph transform $\Gamma.$ Next, will be given a brief outline of the idea behind the graph transform $\Gamma$, which is also illustrated in Figure \ref{IF}.
\begin{figure}[h]
\centering
\includegraphics[height=8cm,width=10cm]{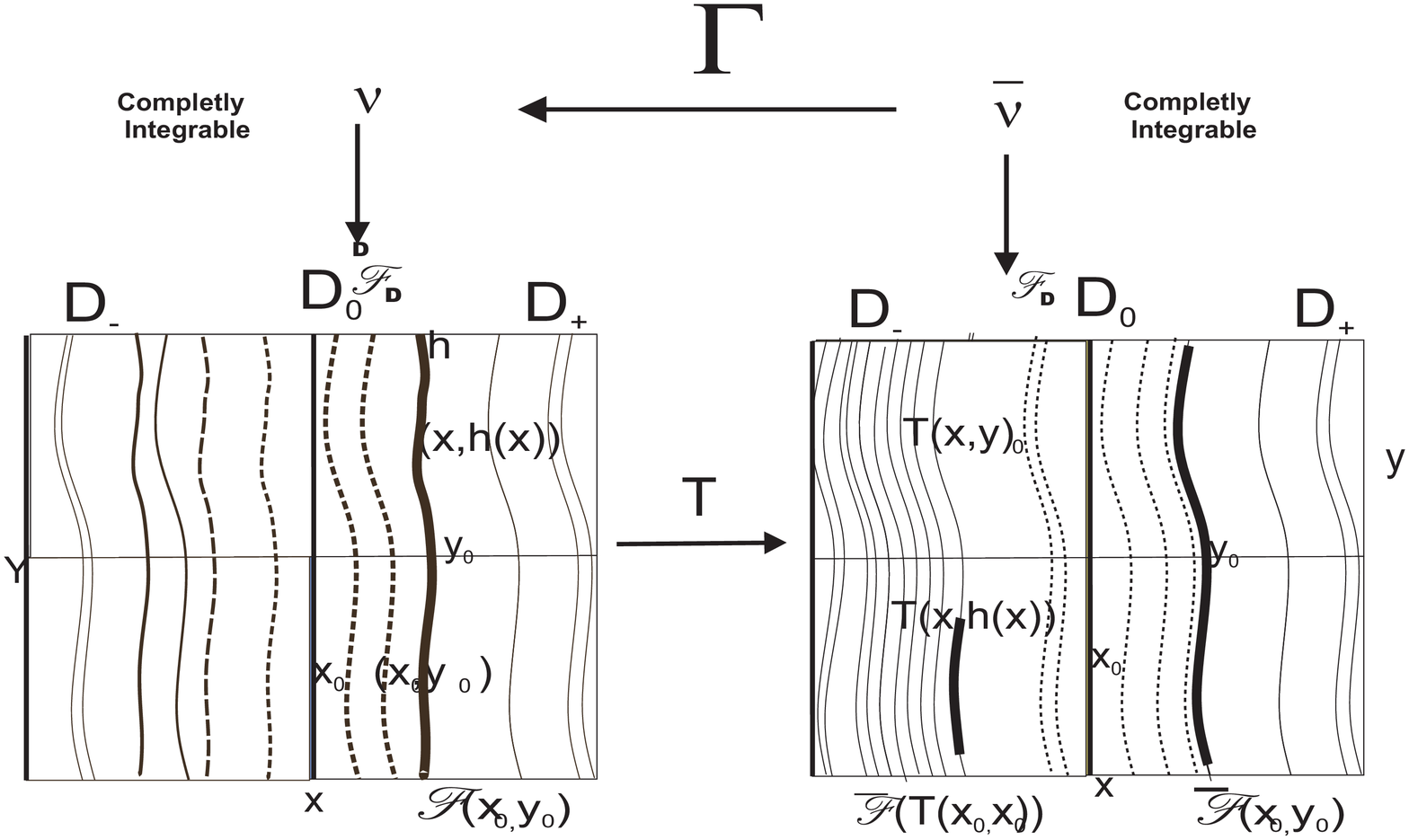}
 \caption{Graph transform $\Gamma.$}
 \label{IF}
\end{figure}
Our goal is to find a $C^{k}$ integrable function $\nu^{*}:D\subset \R^{n+1} \to \R^{n},$ so that for every integral surface $h$ its graphs is invariant under $T(x,y)=(F(x,y),G(x,y)):=(\overline{x},\overline{y}),$ which means that
\begin{eqnarray*}\label{Fol2S01}
F(x,h(x))&=& \overline{x} , \nonumber\\
G(x,h(x))&=& \overline{h}(\overline{x}),
\end{eqnarray*}
where $\overline{h}$ is an integral surface of $\nu^{*}.$ To find $\nu^{*}$ we take any completely integrable function $\overline{\nu}:D \to \R^{n}$  and seek a completely integrable function $\nu:D\to \R^{n}$ so that
\begin{eqnarray*}\label{Fol2S011}
F(x,h(x))&=& \overline{x} , \\
G(x,h(x))&=& \overline{h}(\overline{x}),
\end{eqnarray*}
where $h$ is an integral surface of $\nu$ and $\overline{h}$ is an integral surface of $\overline{\nu}.$\\
If such a function exists, we define the graph transform of $\overline{\nu}$ via $\Gamma(\overline{\nu}):=\nu$ and note that the desired function $\nu^{*}$ is a fixed point of the graph transform so that $\Gamma(\nu^{*})=\nu^{*}.$  It is not difficult to see that
\begin{equation*}\label{E:1}
\Gamma(\overline{\nu})(x,y)=\begin{cases}
\displaystyle \frac{\overline{\nu} \circ T(x,y)\partial_{y}G(x,y) -\partial_{y}F(x,y)}{\partial_{x}F(x,y)-\overline{\nu} \circ T(x,y)\partial_{x}G(x,y)}, & y \neq 0,\\
0, & y=0.
\end{cases}
\end{equation*}
%\end{list}

Notice that, in view of Definition \ref{A,B,C}, we can rewrite the operator $\Gamma$ in the following way:
\begin{equation*}
 \displaystyle\Gamma(\overline{\nu})(x,y)=\begin{cases}
\displaystyle\frac{(\overline{\nu}\circ TA-C)}{(1-\overline{\nu} \circ T B)}(x,y),& y\neq 0,\\
0, &y=0.
\end{cases}
\end{equation*}
It is not difficult to show that the graph transform $\Gamma$ is well defined on a complete sub-space $\mathcal{A}_{L}$ of the continuous function from $D$ to $\R^{n},$ and that $\Gamma$ has a fixed point $\nu^{*}$ (see Theorem \ref{fgamma}). Our goal in this work is to show that the fixed point $\nu^{*}$ is a $C^{k}$ function completely integrable. Then, by using the Idea 1 we have that the graphs of the integral surfaces give the foliation $\mathcal{F}_{D}$ of Theorem \ref{MT}.
\subsection{The Operator $\Gamma$}\label{OPG}
Our goal in this section is to give a rigorous definition and state some properties of the operator $\Gamma$ described informally in the last subsection. We begin by introducing the following definition.
\begin{definicao}\label{DAL}
Let $L \geq 0.$ We denote by $\mathcal{A}_{L}$  the set of all the functions $\nu:D\to \R^{n} $ which satisfies the following conditions:
\begin{itemize}
\item[$(a)$] $\nu$ is continuous on $D;$
\item[$(b)$]$\|\nu\|\leq L;$
\item[$(c)$] $\nu(x,0)=0,$  \mbox{if} $\|x\|\leq1.$
\end{itemize}
\end{definicao}
\begin{observacao}\label{DALC}
Since $\R^{n}$ is a complete normed space, it is not difficult to show that $\mathcal{A}_{L}$ is a complete metric space with the norm of the supremum.
\end{observacao}
Now we are ready to define the most important operator of our work. This operator is denoted by $\Gamma$ and is defined as in ~\cite[Eq. (6)]{Shashkov} by
 \begin{definicao}\label{gamma}
  \begin{equation}
 \begin{array}{cccc}
\Gamma \ : & \! \mathcal{A}_{L} & \! \longrightarrow
& \! \Gamma(\mathcal{A}_{L}) \\
& \! \overline{\nu }& \! \longmapsto
& \! \nu=\Gamma(\overline{\nu}),
\end{array}
\end{equation}
where the function $\Gamma(\overline{\nu}):D \to \R^{1\times n}$ is given by
\begin{equation*}
 \displaystyle\Gamma(\overline{\nu})(x,y)=\begin{cases}
\displaystyle\frac{(\overline{\nu}\circ TA-C)}{(1-\overline{\nu} \circ T B)}(x,y),& y\neq 0,\\
0, &y=0,
\end{cases}
\end{equation*}
with the functions $A,B$ and $C$ as in Definition \ref{A,B,C}.
\end{definicao}

Next, we list a few basic properties of the operator $\Gamma.$ Details may be found in ~\cite[Lemma 1]{Shashkov} or \cite[Proposition 3.17]{vidarte}.
\begin{prop}\label{fgamma}
There is a constant  $L\geq 0$ such that
\begin{itemize}
  \item[$(a)$]$\Gamma(\mathcal{A}_{L})\subset \mathcal{A}_{L}.$
  \item[$(b)$] The operator $\Gamma:\mathcal{A}_{L}\to \mathcal{A}_{L}$ is a contraction.
  \item[$(c)$] The operator $\Gamma:\mathcal{A}_{L} \to \mathcal{A}_{L}$ has a unique fixed point $\nu^{*}$ completely integrable function.
  \item[$(d)$]The operator $\Gamma$ takes completely integrable function  into completely integrable function. Moreover, if $\overline{\mathcal{F}}_{D}$ and $\mathcal{F}_{D}$ are foliations defined by the completely integrable functions $\overline{\nu}$ and $\Gamma(\overline{\nu})$ respectively; then $T$ takes every leaf $\mathcal{F}_{(x_{0},y_{0})} \in \mathcal{F}_{D},$ $\mathcal{F}_{(x_{0},y_{0})}\neq D_{0},$ into a part of the leaf $\overline{\mathcal{F}}_{T(x_{0},y_{0})} \in \overline{\mathcal{F}}_{D}$, that is, $T(\mathcal{F}_{(x_{0},y_{0})})\subset \overline{\mathcal{F}}_{T(x_{0},y_{0})}$.
  \end{itemize}
\end{prop}
\begin{observacao}
It is known from ~\cite[Eq. 8 1]{Shashkov} and \cite[Eq. 3.47]{vidarte}  that $L$  can be taken as
\begin{equation}\label{L}
L=\frac{-(1-\norm{A})+\sqrt{(\norm{A}-1)^{2}-4\norm{B}\norm{C}}}{2\norm{B}}.
\end{equation}

\end{observacao}
Let us begin stating the main proposition of this article.
 \begin{prop}\label{Mp01}
 Let $L \geq 0$ be as in Proposition \ref{fgamma}. Then, the attracting fixed point $\nu^{*}$ of the operator $\Gamma$ is a function of class $C^{k}.$
\end{prop}
The proof of this proposition will be given with the following propositions which will be proven  in the next sections.
\begin{prop}\label{Di1}
If $\mu \in \mathcal{A}_{L}$ is a $C^{k}$ function. Then, the following statements hold:
\begin{itemize}
\item[(a)]$ \lim_{(a,b)\to (x,0)}D^{i}(\Gamma(\mu))(a,b)=0,$ for all $1\leq i \leq k$ and $(x,0)\in D_{0}.$
\item[(b)] The function $\Gamma(\mu) \in \mathcal{A}_{L}$ is of class $C^{k}$ and $D^{i}\Gamma(\mu)(x,0)=0,$ for all  $1\leq i \leq k$ and  $(x,0)\in D_{0}.$
\end{itemize}
\end{prop}
\begin{prop}\label{propstep2}
If $\overline{\nu} \in \mathcal{A}_{L}$ is a $C^{k}$ function and  $D^{i}\overline{\nu}(x,0)=0,$ for all $0\leq i \leq k$ and $(x,0)\in D_{0}.$ Then, the following limit  exists
 \begin{equation*}\label{ME:3}
 \displaystyle \lim_{n\to \infty}(\Gamma^{n}(\overline{\nu}),D(\Gamma^{n}(\overline{\nu})),\ldots,D^{k}(\Gamma^{n}(\overline{\nu})))=(\nu^{*}, A_{1},A_{2},\ldots,A_{k}),
 \end{equation*}
 where $A_{1},A_{2},\ldots,A_{k}$ are continuous functions.
\end{prop}
\begin{proof}[\textbf{Proof of Proposition \ref{Mp01}}]
%Choose a function $\overline{\nu} \in \mathcal{A}_{L}$ of class $C^{k}$. By induction it follows that
%\begin{equation*} \label{l4} \widetilde{N}_{i}^{n}(\overline{\nu},D\overline{\nu},\ldots,D^{i}\overline{\nu})=(\Gamma^{n}(\overline{\nu}),D(\Gamma^{n}(\overline{\nu})),\ldots,D^{i}(\Gamma^{n}(\overline{\nu}))).
%\end{equation*}
%Hence, on account of Proposition \ref{l3} one obtains
% \begin{equation*}\label{ME:3}
% \displaystyle \lim_{n\to \infty}(\Gamma^{n}(\overline{\nu}),D(\Gamma^{n}(\overline{\nu})),\ldots,D^{k}(\Gamma^{n}(\overline{\nu})))=(\nu^{*}, A_{1},A_{2},\ldots,A_{k}),
% \end{equation*}
Let $\nu \in \mathcal{A}_{L}$ be a $C^{k}$ function. By Proposition \ref{Mp01} we have that $\overline{\nu}:=\Gamma(\nu)$ is a $C^{k}$ function and that $D^{k}\overline{\nu}(x,0)=0,$ From Proposition \ref{propstep2}, we have that
 \begin{equation*}\label{ME:3}
 \displaystyle \lim_{n\to \infty}(\Gamma^{n}(\overline{\nu}),D(\Gamma^{n}(\overline{\nu})),\ldots,D^{k}(\Gamma^{n}(\overline{\nu})))=(\nu^{*}, A_{1},A_{2},\ldots,A_{k}).
 \end{equation*}
Hence, and by using interchanging the order of differentiation and limit, see \cite[Theorem 8.6.3]{Dieudonne}, we obtain
$D^{j}(\displaystyle \lim_{n\to \infty}\Gamma^{n}(\overline{\nu}))=A_{j},$
for $0\leq j \leq k.$ Thus, since $\nu^{*}$ is a global attracting fixed of $\Gamma,$ it follows that $D^{j}(\nu^{*})=A_{j},$
%\begin{equation*}\label{ME:6}
%D^{j}(\nu^{*})=A_{j},
%\end{equation*}
for $0\leq j \leq k.$
Therefore, since $A_{j}$ is a continuous function, it follows that the function $\nu^{*}$ is of class $C^{k},$ which concludes the proof of our main proposition.

\end{proof}
Now we are ready to prove our main Theorem \ref{MT}.
\begin{proof}[\textbf{Proof of Theorem \ref{MT}}]
By Proposition \ref{fgamma}(c) we have that the attracting fixed point $\nu^{*}$ of the operator $\Gamma$ is integrable and by Prop \ref{Mp01} we get that   $\nu^{*}$ is of class $C^{k}$. Thus, the function $\nu^{*}$ defines a foliation $\mathcal{F}_{D}$ of class $C^{k}$ and by Proposition \ref{fgamma}(d) it follows that the foliation $\mathcal{F}_{D}$ is $T$-invariant, which finishes the proof the of our  main result.
\end{proof}
\section{Proof of Proposition \ref{Di1} }\label{PT}

The proof is somewhat lengthy, so we divide it into two parts. \textbf{In the first part:} we will establish a formula for the \emph{$kth$  order derivatives} of the function $\Gamma(\nu)$ at the points $(x,y),$ where the $y$-component stay away from zero . \textbf{In the second part:} we \emph{estimate the norms of the $ith$ derivatives} of the functions $A(x,y),$ $B(x,y)$ and $C(x,y),$ at the points $(x,y)$ around of a neighborhood of $D_{0}.$

%\newcounter{quest1}
%\begin{list}{\textbf{Step 1.\arabic{quest1}:}}{\usecounter{quest1}
%\setlength{\labelwidth}{-2mm} \setlength{\parsep}{0mm}
%\setlength{\topsep}{0mm} \setlength{\leftmargin}{0mm}}
%\renewcommand{\labelenumi}{(\alph{enumi})}
%\item To establish a formula for the $kth$  order derivatives of the function $\Gamma(\nu)$ at the points $(x,y),$ with $y\neq0.$
%\item To estimate the norms of the $ith$ derivatives of the functions $A(x,y),$ $B(x,y)$ and $C(x,y),$ at the points $(x,y)$ around of a neighborhood of $D_{0}.$
%\end{list}

\subsection {Part 1: Formula for Derivatives } Before that, we introduce some definitions which will be useful in order to find suitable formulas. From now on, $L(E_{1},\ldots,E_{k};G)$ denotes the space of continuous $k$-multilinear maps of  $E_{1},\ldots,E_{r}$ to $G.$ If $E_{i}=E,i \leq k,$ this space is denoted $L^{k}(E,F).$ Moreover, $L_{s}^{k}(E; F)$ denotes the subspace of symmetric elements of $L^{k}(E,F).$

\begin{definicao}[Symmetrizing operator] The Symmetrizing operator $Sym^{k}$ is defined by
$$
\begin{array}{cccc}
Sym^{k} \ : & \! L^{k}(E; F) & \! \longrightarrow
& \! L^{k}(E; F) \\
& \! A & \! \longmapsto
&\displaystyle \! Sym^{k}(A)=\frac{1}{k!}\sum_{\sigma \in S_{k}}\sigma A,
\end{array}
$$
where $(\sigma A)(e_{1},\ldots, e_{k})=A(e_{\sigma(1)},\ldots, e_{\sigma(k)})$ and $S_{k}$ is the group of permutations on $k$ elements.
\end{definicao}
\begin{observacao}\label{RP}
The symmetrizing operator $Sym^{k}$ satisfies the following properties:
\begin{itemize}
\item[(a)] $Sym^{k}(L^{k}(E; F))=L_{s}^{k}(E; F),
$
%\footnote{We use the notation $L_{s}^{k}(E; F)$ for the space of symmetric continuous $k$-multilinear maps from $E^{k}$ to $F$.},
\item[(b)]$(Sym^{k})^{2}=Sym^{k},$
\item[(c)]$\parallel Sym^{k} \parallel\leq1.$
\end{itemize}
\end{observacao}

\begin{definicao}\label{DCP0}
Assume $\mathcal{B} \in L(F_{1}\times F_{2}; G),$ we define the bilinear map.
$$
\begin{array}{ccccc}
  \phi^{(i,(k-i))} & :&L^{i}(E;F_{1})\times L^{(k-i)}(E;F_{2}) & \longrightarrow & L^{k}(E;G) \\
   &  &   (A_{1},A_{2})&\longmapsto & [\phi^{(i,(k-i))}(A_{1},A_{2})]
\end{array}
$$
 by
\begin{equation}
[\phi^{(i,(k-i))}(A_{1},A_{2})](e_{1},\ldots,e_{k})=\mathcal{B}(A_{1}(e_{1},\ldots e_{i}),A_{2}(e_{i+1},\ldots,e_{k})).
\end{equation}
\end{definicao}
\begin{definicao}\label{DCP1}
Let $\overline{\nu}_{i}:U \to L^{i}(E;F_{1})$ and $\overline{\nu}_{(k-i)}:U \to L^{(k-i)}(E;F_{2}),$ we define
\begin{equation}\label{P1}
\begin{array}{ccccc}
  \phi^{(i,(k-i))}(\overline{\nu}_{i},\overline{\nu}_{(k-i)}) & : & U & \mapsto & L^{k}(E;G) \\
   &  & p & \to & \phi^{(i,(k-i))}(\overline{\nu}_{i}(p),\overline{\nu}_{(k-i)}(p)).
\end{array}
\end{equation}
\end{definicao}
\begin{definicao}\label{MM} For every tuple $(q,r_{1},r_{2},\ldots,r_{q}),$ where $q>1,$ and $r_{1}+\ldots+r_{q}=k,$ we define the following continuous multilinear map
    \begin{equation}
    \begin{array}{cccll}
          \phi^{(q,r_{1},\ldots,r_{q})} & :  L^{q}(F;G)\times L^{r_{1}}(E;F)\times \ldots \times L^{r_{q}}(E;F)  \longrightarrow  L^{k}(E;G) \\
           &   (\overline{v}_{q},\overline{v}_{r_{1}},\ldots,\overline{v}_{r_{q}}) \longmapsto  \phi^{(q,r_{1},\ldots,r_{q})}(\overline{v}_{q},\overline{v}_{r_{1}},\ldots,\overline{v}_{r_{q}}),
        \end{array}
        \end{equation}
        where $$ \phi^{(q,r_{1},\ldots,r_{q})}(\overline{v}_{q},\overline{v}_{r_{1}},\ldots,\overline{v}_{r_{q}}):\underbrace{E\times \ldots \times E}_{\mbox{k-times}}\to G$$ is defined as \\
     \begin{multline}
     \phi^{(q,r_{1},\ldots,r_{q})}(\overline{v}_{q},\overline{v}_{r_{1}},\ldots,\overline{v}_{r_{q}})(e_{1},\ldots,e_{k})\\=\overline{v}_{q}(\overline{v}_{r_{1}}(e_{1},\ldots e_{j_{r_{1}}}),\ldots, \overline{v}_{r_{q}}(e_{(j_{r_{1}}+j_{r_{2}}+\ldots+j_{r_{(q-1)}}+1)},\ldots,e_{(j_{r_{1}}+j_{r_{2}}+\ldots+j_{r_{q}})}).
     \end{multline}
\end{definicao}
 \begin{definicao}\label{MM1}
 Let $U \subset E$ such that $\overline{\nu}_{r_{i}}:U \to L^{r_{i}}(E;F),1\leq i \leq q$ and $f:V\subset U \to U$ are functions. Then, define
    \begin{equation}
    \phi^{(q,r_{1},\ldots,r_{q})}* ((\overline{\nu}_{q}\circ f)  \times \overline{\nu}_{r_{1}}\times\ldots\times \overline{\nu}_{r_{q}})  : U \rightarrow  L^{k}(E;G)
    \end{equation}
     by  $$u  \to  \phi^{(q,r_{1},\ldots,r_{q})}((\overline{\nu}_{q}\circ f(u)),\overline{\nu}_{r_{1}}(u),\ldots,\overline{\nu}_{r_{q}}(u)),$$
    where $\phi^{(q,r_{1},\ldots,r_{q})}((\overline{\nu}_{q}\circ f(u)),\overline{\nu}_{r_{1}}(u),\ldots,\overline{\nu}_{r_{q}}(u))$ as in Definition \ref{MM}.
 \end{definicao}
Next, we define generalizations of the $kth$ derivative of the composition of two functions.
 \begin{definicao}\label{DC}
 Let $k_{1}\geq k_{3}\geq k_{2} \geq 1$ be integers such that  we have the functions $\overline{v}_{q}:V \subset F\to L^{q}(F,G),$ for $ k_{2}\leq q \leq k_{3}$ and that $f:U \to V$ is a function of class $C^{k_{1}-k_{2}+1},$ where $D^{i}f:U \to L^{i}(E,F), 0\leq i \leq k_{1}-k_{2}+1$ are the derivatives of $f.$ Then, we define the function
 $$\mathcal{DC}^{(k_{1},k_{2},k_{3})}((\overline{v}_{k_{2}},\ldots,\overline{v}_{k_{3}}),f):U \to L^{k_{1}}(E,F)$$
 given by
 \begin{multline}
\mathcal{DC}^{(k_{1},k_{2},k_{3})}((\overline{v}_{k_{2}},\ldots,\overline{v}_{k_{3}}),f)(p)\\ :=\displaystyle  Sym^{k_{1}} \left(\displaystyle\sum_{n=k_{2}}^{k_{3}}\sum_{{r_{1}+\ldots+r_{n}=k_{1}}} \frac{k_{1}!}{r_{1}!\ldots r_{n}!} \phi^{(n,r_{1},\ldots,r_{n})}*((\overline{v}_{n}\circ f)  \times    D^{r_{1}}f\times\ldots\times D^{r_{n}}f)\right)(p),
\end{multline}
where $ \phi^{(n,r_{1},\ldots,r_{n})}*((\overline{v}_{n}\circ f)  \times \overline{\nu}_{r_{1}}\times\ldots\times \overline{\nu}_{r_{n}})$ as in Definition \ref{MM1}, and when $k_{1}=k_{3}=0$ and $k_{2}=1,$ we define the function
$$\mathcal{DC}^{(0,1,0)}(\overline{v}_{0},f):U \to F$$
given by
\begin{equation}\label{DC1}
\mathcal{DC}^{(0,1,0)}(\overline{v}_{0},f)(p):=(\overline{v}_{0}\circ f)(p).
\end{equation}
Furthermore, if $\overline{\nu}:V\subset F \to G$ and $f:U \subset E\to V\subset F$ are $C^{k_{3}},$ then, will be used the following notation.
\begin{equation}\label{N4.2}
\mathcal{DC}^{(k_{1},k_{2},k_{3})}(\overline{\nu},f):=\mathcal{DC}^{(k_{1},k_{2},k_{3})}(D^{k_{2}}(\overline{\nu}),\ldots,D^{k_{3}}(\overline{\nu}),f).
\end{equation}
\end{definicao}
\begin{observacao}\label{DCE}
If  $\overline{\nu}:D\to \R^{1\times n}$ and $T: D^{*} \to D $ are functions $C^{k},$ then:
\begin{itemize}
\item[(i)]on account of the chain rule applied to the function $\overline{\nu} \circ T$ it is possible to show that
\begin{equation*}\label{Dk1k}
\mathcal{DC}^{(k,1,k)}(\overline{\nu},T):=D^{k}(\overline{\nu}\circ T).
\end{equation*}
Therefore, we conclude that the function in Definitions \ref{DC} is a genera\-lization of the $k$th derivative of the composite of two functions.
\item[(ii)] By Eq. \eqref{N4.2} and the symmetry of the function $D^{k}\overline{\nu}$ , we obtain
\begin{equation}\label{Dkkk}
\mathcal{DC}^{(k,k,k)}(\overline{\nu},T):=k!(D^{k}\overline{\nu}) \circ T \underbrace{DT\ldots DT}_{k-times}.
\end{equation}
\end{itemize}
\end{observacao}

 Next, we define generalizations of the $kth$ derivative of product of the map $(f\circ g)$ with $h.$
\begin{definicao}\label{DCP2}
Assume $\mathcal{B} \in L(F_{1}\times F_{2};G)$  and that $k_{1}\geq k_{3} \geq 1$; $k_{2}\geq 0$ are integers such that  $f:U \subset E \to V\subset F$ and $B:U\to F_{2}$ are  functions of class $C^{k_{3}}$ and $C^{k_{1}-k_{2}}$ respectively, where  $D^{i}B:U \to L^{i}(E,F_{2}),$ for $0\leq i \leq k_{1}-k_{2}$ are the derivatives of the function $B,$ moreover consider the functions $\overline{\nu}_{i}:V \subset F \to L^{i}(F,F_{1}),$ $0\leq i\leq k_{3}.$ Then, we define the map
 $$\displaystyle\mathcal{DCP}^{(k_{1},k_{2},k_{3})}(f,(\overline{\nu}_{0},\overline{\nu}_{1},\ldots,\overline{\nu}_{k_{3}}),B):V \to L^{k_{1}}(E;G),$$
given by
\begin{multline}
\displaystyle\mathcal{DCP}^{(k_{1},k_{2},k_{3})}(f,(\overline{\nu}_{0},\overline{\nu}_{1},\ldots,\overline{\nu}_{k_{3}}),B)(p)
:= \\ \displaystyle Sym^{k_{1}}\left(\sum_{n=k_{2}}^{k_{3}}{k_{1} \choose n}\phi^{(n,k_{1}-n)}\left(\mathcal{DC}^{(n,1,n)}\displaystyle \left((\overline{\nu}_{1},\ldots,\overline{\nu}_{n}),f\right),D^{k_{1}-n}B\right)\right)(p),\\
\end{multline}
where $\phi^{(n,k_{1}-n)}$ as in Definition \ref{DCP1}. Moreover, if $\overline{\nu}:U \subset E \to  F$ is  a function of class $C^{k_{3}},$ then we will use the notation
\begin{equation}\label{N6.3}
\displaystyle\mathcal{DCP}^{(k_{1},k_{2},k_{3})}(\overline{\nu},f,B):=\displaystyle\mathcal{DCP}^{(k_{1},k_{2},k_{3})}(f,(\overline{\nu},D(\overline{\nu}),\ldots, D^{k_{3}}(\overline{\nu})),B).
\end{equation}
\end{definicao}
\begin{observacao}\label{DCPE}
 Let $F_{1}$ and $F_{2}$ be the space of the $n$-columns and $n$-rows respectively. Then, we define the multilinear map $\mathcal{B}:F_{1}\times F_{2}\to \R$ given by $\mathcal{B}(A,B)=A\times B,$ where $A\times B$ is the usual product of matrices. Assume that  $\overline{\nu}:D \to \R^{n},$  $T:D^{*} \to D$ and $B:D^{*} \to F_{1}$ are  $C^{k}$ functions. Then, by using Leibniz and chain rule applied to the functions $\left((\overline{\nu} \circ T).B\right)$ and $(\overline{\nu} \circ T),$ respectively and in view of Eq. \eqref{Dk1k} and Definition \ref{DCP2} it is easy to show
\begin{equation}\label{DCPE1}
\displaystyle \mathcal{DCP}^{(k,0,k)}(\overline{\nu},T,B):=\displaystyle D^{k}((\nu \circ T)B).
\end{equation}
This show that the function in  Definition \ref{DCP2} generalizes the $k$th derivative of the product of the map $(\nu \circ T)$ with $B$. This fact will be useful later.
\end{observacao}
Next, we define generalizations of the $k$th derivative of the map $(1-\nu \circ T B)^{-1},$ where $B$ is a function as in Definition \ref{A,B,C} and $\nu \in \mathcal{A}_{L}$ is a function of class $C^{k}.$
\begin{definicao}\label{DCPI}
Let $(q,r_{1},\ldots,r_{q},r)$ be a tuple with $q \geq 1,$ $r_{1}+\ldots +r_{q}=k$ and $r=\max\{r_{1},\ldots, r_{q}\}$ such that   $ \overline{\nu}_{i}:V\subset F \to L_{s}^{i}(F,F_{1}),$ $0\leq i\leq r$ are functions and that $T:U \subset E\to V \subset F$ and $B:U\to F_{2} $ are functions of class $C^{k}.$ Then, we define the map
 $$\displaystyle \prod^{(r_{1},\ldots,r_{q}, r)}(\overline{\nu}_{0},\ldots,\overline{\nu}_{r_{q}},T,B):U \to L^{k}(E,G)$$
 given by
\begin{multline}
\prod^{(r_{1},\ldots,r_{q}, r)}(\overline{\nu}_{0},\ldots,\overline{\nu}_{r_{q}},T,B):
=\\  \mathcal{DCP}^{(r_{1},0,r_{1})}(T,(\overline{\nu}_{0},..,\overline{\nu}_{r_{1}}),B)\times .. \times\mathcal{DCP}^{(r_{q},0,r_{q})}((\overline{\nu}_{0},..,\overline{\nu}_{r_{q}}),T,B),
\end{multline}
where $\mathcal{DCP}^{(r_{i},0,r_{i})}(T,(\overline{\nu}_{0},..,\overline{\nu}_{r_{i}}),B)), 1\leq i \leq q$ as in Definition \ref{DCP2}.
\end{definicao}
\begin{definicao}\label{DCPI1}
 Under the notations of Definition \ref{DCPI}. Suppose that $k_{1}\geq k_{3}\geq k_{2} \geq 1$ are integers . Then, we define the map
 $$\mathcal{DICP}^{(k_{1},k_{2},k_{3})}(T,(\overline{\nu}_{0},\ldots,\overline{\nu}_{(k_{1}-k_{2})+1}),B):U \to L^{k_{1}}(E,G)$$
 given by
\begin{multline}
\mathcal{DICP}^{(k_{1},k_{2},k_{3})}((\overline{\nu}_{0},\ldots,\overline{\nu}_{(k_{1}-k_{2})+1}),T,B)
\\:=Sym^{k_{1}}\left(\sum_{q=k_{2}}^{k_{3}}\displaystyle\sum_{r_{1}+\ldots+r_{q}=k_{1}}\frac{k!(-1)^{q}q!}{r_{1}!\ldots r_{q}!}(1-\nu_{0}\circ TB)^{-(q+1)}\prod^{(r_{1},\ldots,r_{q}, r)}(\overline{\nu}_{0},\ldots,\overline{\nu}_{r_{q}},T,B)\right),
\end{multline}
where $\displaystyle \prod^{(r_{1},\ldots,r_{q}, r)}(\overline{\nu}_{0},\ldots,\overline{\nu}_{r_{q}},T,B)$ as in Definition \ref{DCPI}. Furthermore, if $\overline{v}:U \subset F \to  F_{1}$ is  a $C^{k_{1}-k_{2}+1}$ map will be used the following notation
\begin{equation}\label{DCPI2}
\displaystyle\mathcal{DICP}^{(k_{1},k_{2},k_{3})}(\overline{\nu},T,B):=\displaystyle\mathcal{DICP}^{(k_{1},k_{2},k_{3})}((\overline{\nu},D(\overline{\nu}),\ldots, D^{k_{1}-k_{2}+1}(\overline{\nu})), T,B).
\end{equation}
\end{definicao}

\begin{observacao}\label{DCPI4}
 Let $\nu \in \mathcal{A}_{L}, B, T$ be maps as in Definition \ref{DAL}, Definition \ref{A,B,C} and Definition \ref{T0} respectively such that $\nu$ and $B$ are $C^{k}$. Define $\mathcal{I}:\R-\{0\}\to \R$ by $\mathcal{I}(x)=1/x.$ Since, the function $(1-\overline{\nu} \circ TB):D^{*}\to \R$ is nonzero, it follows from chain rule applied to $\mathcal{I} \circ (1-\overline{\nu} \circ TB)$ and Definition \ref{DCPI1} that
\begin{eqnarray}\label{inv0.1}
D^{k}(1-\overline{\nu} \circ TB)^{-1}&=&Sym^{k}\circ\sum_{q=1}^{k}\displaystyle\sum_{{r_{1}+\ldots+r_{q}=k}}\frac{k!q!}{r_{1}!\ldots r_{q}!}\frac{ D^{r_{1}}(\overline{\nu} \circ TB)\ldots D^{r_{q}}(\overline{\nu} \circ TB)}{(1-\overline{\nu} \circ TB)^{(q+1)}} \nonumber\\
&:=& \mathcal{DICP}^{(k,1,k)}(\overline{\nu},T,B).
\end{eqnarray}
This show that the function in Definition \ref{DCPI1} generalizes the $k$th derivative of the map $(1-\nu \circ B)^{-1}.$ This fact will be useful later.
\end{observacao}
Now, using the last definitions we find one formula for the $kth$ derivative of the function $\Gamma(\nu)$ at the points $(x,y),$ with $y\neq0$ (see Lema \ref{MC}). This generalizes the formulas given in \cite[Eq. (11)]{Shashkov} and \cite[Eq. (42)]{Pacifico}, it is quite important to prove our main Proposition \ref{Mp01}. We start by noticing the following simple but very useful lemma.
\begin{lema}\label{DGamma}
Under Definitions \ref{A,B,C}, \ref{DAL} and $\ref{gamma}.$ Assume that $\overline{\nu} \in \mathcal{A}_{L}$ is a $C^{k}$ function. Then, for $y \neq 0$ the following formulas hold:
\begin{eqnarray}\label{k=1}
D(\Gamma(\overline{\nu}))(x,y)& =&(\overline{\nu}\circ TA-C)D(1-\overline{\nu} \circ T B)^{-1}(x,y)\\
&& + D(\overline{\nu}\circ TA-C)(1-\overline{\nu}\circ T B)^{-1}(x,y) \nonumber\\
&:=&\PC{U_{1}^{1}(\overline{\nu},T,A,B,C)+U_{2}^{1}(\overline{\nu},T,A,B,C)}(x,y)\nonumber;
\end{eqnarray}
for $k\geq 2$
\begin{eqnarray}\label{kgeq2}
D^{k}(\Gamma(\overline{\nu}))(x,y)&=& Sym^{k}\circ \overline{\nu}\circ TA-CD^{k}(1-\overline{\nu} \circ T B)^{-1}(x,y)\\
&+&Sym^{k}\circ D^{k}(\overline{\nu}\circ TA-C)(1-\overline{\nu}\circ T B)^{-1}(x,y)\nonumber\\
&+&Sym^{k}\circ \sum_{q=1}^{k-1}{k \choose q}D^{q}(\overline{\nu}\circ TA-C)D^{k-q}(1-\overline{\nu}\circ TB)^{-1}(x,y). \nonumber\\
&:=&Sym^{k}\circ \left(U_{1}^{k}(\overline{\nu},T,A,B,C)+U_{2}^{k}(\overline{\nu},T,A,B,C)\right)(x,y) \nonumber \\
&+&Sym^{k}\circ\left(U_{3}^{k}(\overline{\nu},T,A,B,C)\right)(x,y)\nonumber.
\end{eqnarray}
\end{lema}
\begin{proof}
This is a direct consequence of Leibnitz's rule.

\end{proof}

\begin{lema}\label{U1}
Under Definitions \ref{A,B,C} and \ref{DAL}. Assume that $\overline{\nu} \in \mathcal{A}_{L}$ is  a $C^{k}$ function and that
$U_{1}^{k}(\overline{\nu},T,A,B,C):D^{*} \to L^{k}(\R^{n+1},\R^{n})$ is as in Lemma \ref{DGamma}. Then the following formulas hold:
%, that is,
% \begin{equation}\label{DU1}
%U_{1}^{k}(\overline{\nu},T,A,B,C)(x,y):=(\overline{\nu}\circ TA-C)D^{k}(1-\overline{\nu} \circ T B)^{-1}(x,y).
% \end{equation}  Then the following formulas hold:
\begin{equation}\label{11.F}
 U_{1}^{1}(\overline{\nu},T,A,B,C)=(\overline{\nu}\circ TA-C)(1-\overline{\nu} \circ T B)^{-2}\left(\overline{\nu}\circ T DB+D \overline{\nu}\circ T.DT.B\right),
 \end{equation}
 for $k\geq 2$
\begin{eqnarray}\label{1k.F}
U^{k}_{1}(\overline{\nu},T,A,B,C)&=&(\overline{\nu}\circ TA-C)k!(1-\overline{\nu}\circ TB)^{-2}Sym^{k}\circ \left(\mathcal{DC}^{(k,k,k)}(\overline{\nu},T)B\right) \nonumber \\
&+&(\overline{\nu}\circ TA-C)k!(1-\overline{\nu}\circ TB)^{-2}Sym^{k}\circ \left(\mathcal{DC}^{(k,1,(k-1))}(\overline{\nu},T)\right) \nonumber\\ &+&(\overline{\nu}\circ TA-C)k!(1-\overline{\nu}\circ TB)^{-2}Sym^{k}\circ \left(\mathcal{DCP}^{(k,0,(k-1))}(\overline{\nu},T,B)\right)\nonumber \\
&+&\displaystyle (\overline{\nu}\circ TA-C)\mathcal{DICP}^{(k,2,k)}(\overline{\nu},T,B),
\end{eqnarray}
where $\mathcal{DC}^{(k_{1},k_{2},k_{3})}(\overline{\nu},T),$ $\mathcal{DCP}^{(k_{1},k_{2},k_{3})}(\overline{\nu},T,B)$ and $\mathcal{DICP}^{(k_{1},k_{2},k_{3})}(\overline{\nu},T,B))$ as in Definitions \ref{DC}, \ref{DCP2} and \ref{DCPI1},
 respectively.
\end{lema}
\begin{proof}
We do the proof for the formula  \eqref{1k.F}. The proof of the formula \eqref{11.F} is straightforward. From assumption and Remark \ref{DCPI4} it follows that
\begin{eqnarray}\label{U1k}
U_{1}^{k}(\overline{\nu},T,A,B,C)&=&(\overline{\nu}\circ TA-C)D^{k}(1-\overline{\nu} \circ T B)^{-1}.\nonumber\\
&:=&(\overline{\nu}\circ TA-C)I_{1}^{k}.
\end{eqnarray}
We observe that, for $k \geq 2,$ by the chain rule,
\begin{eqnarray}\label{I1.1}
 I_{1}^{k}&=&k!(1-\overline{\nu}\circ TB)^{-2}Sym^{k}\circ\left(D^{k}(\overline{\nu}\circ TB)\right) \nonumber \\
&+&Sym^{k}\circ \sum_{q=2}^{k}\sum_{r_{1}+\ldots+r_{q}=k}\frac{k!}{r_{1}!\ldots r_{q}!}(1-\overline{\nu}\circ TB)^{-(q+1)}D^{r_{1}}(\overline{\nu}\circ TB)\ldots D^{r_{q}}(\overline{\nu}\circ TB) \nonumber.\\
&:=&k!(1-\overline{\nu}\circ TB)^{-2}Sym^{k}\circ \left(I_{2}^{k}\right)+ Sym^{k}\circ \sum_{q=2}^{k}\sum_{r_{1}+\ldots+r_{q}=k}\frac{k!}{r_{1}!\ldots r_{q}!}(1-\overline{\nu}\circ TB)^{-(q+1)} I_{3}^{k} .
 \end{eqnarray}
  Now applying Leibniz's rule to the function $(\overline{\nu}\circ T)B$ , it follows that
 \begin{eqnarray}\label{I2.1}
 I_{2}^{k}&=&Sym^{k}\circ \left(D^{k}(\overline{\nu}\circ T)B\right)+Sym^{k}\circ \left(\sum_{q=0}^{k-1}{k \choose q}D^{q}(\overline{\nu}\circ T).D^{k-q}(B)\right).\\
 &:=&Sym^{k}\circ \left(I_{2,1}^{k}B\right)+ Sym^{k}\circ \left(\sum_{q=0}^{k-1}{k \choose q}I_{2,2}^{k}\right)\nonumber.
 \end{eqnarray}
Moreover, by using the chain rule to the functions $I_{2,1}^{k}$  and $I_{2,2}^{k}$ respectively, we get
\begin{eqnarray}\label{I4.1}
 I_{2,1}^{k}&=&Sym^{k}\circ \left(k!(D^{k}\overline{\nu}) \circ T \underbrace{DT\ldots DT}_{k-times}\right) \nonumber \\
 &+&Sym^{k}\circ \left( \sum_{q=1}^{k-1}\displaystyle\sum_{{r_{1}+\ldots+r_{q}=k}} \frac{k!}{r_{1}!\ldots r_{q}!} (D^{q}\overline{\nu})\circ T.D^{r_{1}}T \ldots D^{r_{q}}T\right),
 \end{eqnarray}
 and
\begin{equation}\label{I5.F}
I_{2,2}^{k}= \displaystyle Sym^{q}\circ  \left(\sum_{n=1}^{q}\displaystyle\sum_{{r_{1}+\ldots+r_{n}=q}} \frac{q!}{r_{1}!\ldots r_{n}!}(D^{n}\nu)\circ T.D^{r_{1}}T \ldots D^{r_{n}}T\right).D^{k-q}(B).
\end{equation}
Therefore, by replacing $(\ref{I5.F})$ and $(\ref{I4.1})$ into $(\ref{I2.1}),$ and using that $Sym^{k}\circ Sym^{k}=Sym^{k},$ we get
 \begin{eqnarray}\label{I2.2}
 I_{2}^{k}&=&Sym^{k}\left(k!(D^{k}\overline{\nu}) \circ T \underbrace{DT\ldots DT}_{k-times}B\right) \nonumber\\
 &+&Sym^{k}\circ \left( \sum_{q=1}^{k-1}\displaystyle\sum_{{r_{1}+\ldots+r_{q}=k}} \frac{k!}{r_{1}!\ldots r_{q}!} (D^{q}\overline{\nu})\circ F.D^{r_{1}}T \ldots D^{r_{q}}TB\right) \nonumber \\
 &+& Sym^{k}\circ \sum_{q=0}^{k-1}{k \choose q}\displaystyle Sym^{q} \left(\sum_{n=1}^{q}\displaystyle\sum_{r_{1}+\ldots+r_{n}=q} \frac{q!(D^{n}\nu)\circ T.D^{r_{1}}T \ldots D^{r_{n}}T}{r_{1}!\ldots r_{n}!}.D^{k-q}(B)\right).\nonumber \\
 \end{eqnarray}
Hence, on account of Definitions \ref{DC} and \ref{DCP2} we have
 \begin{equation}\label{I2.F}
 I_{2}^{k}:=Sym^{k}(\mathcal{DC}^{(k,1,k)}(\overline{\nu},T)B)+\mathcal{DC}^{(k,1,(k-1))}(\overline{\nu},T,B)+\mathcal{DCP}^{(k,0,(k-1))}(\overline{\nu},T,B).
 \end{equation}
 By similar computation as above, in view of Definitions \ref{DC} and \ref{DCP2} we reach that
\begin{eqnarray}\label{I3.0F}
I_{3}^{k}&:=&\mathcal{DCP}^{(r_{1},0,r_{1})}(\overline{\nu},T,B) \ldots \mathcal{DCP}^{(r_{q},0,r_{q})}(\overline{\nu},T,B).
\end{eqnarray}
Hence, by using Definition \ref{DCP2}, Eq. \eqref{I3.0F} becomes
\begin{equation}\label{I3.F}
I_{3}^{k}=\prod^{(r_{1}, \ldots, r_{q}, r)}(\overline{\nu},T,B).
\end{equation}
Whence, on account of Definition \ref{DCPI1},  we get
\begin{eqnarray}\label{N.0F}
 \displaystyle  Sym^{k}\circ \left(\sum_{q=2}^{k}\sum_{r_{1}+\ldots+r_{q}=k}\frac{k!}{r_{1}!\ldots r_{q}!}(1-\overline{\nu}\circ TB)^{-(q+1)} I_{3}^{k} \right):=\displaystyle \mathcal{DICP}^{(k,2,k)}(\overline{\nu},T,B).
\end{eqnarray}
Thus, by replacing \eqref{N.0F} and $(\ref{I2.F})$ into $(\ref{I1.1}),$ we get
\begin{eqnarray}\label{I1.1F}
I_{1}^{k}&=&Sym^{k}\circ \left(k!(1-\overline{\nu}\circ TB)^{-2}(\mathcal{DC}^{(k,1,k)}((\overline{\nu},T)B))\right)\nonumber\\
&+&Sym^{k}\circ \left(k!(1-\overline{\nu}\circ TB)^{-2}(\mathcal{DC}^{(k,1,(k-1))}(\overline{\nu},T,B))\right)\nonumber\\
&+& Sym^{k}\circ \left(k!(1-\overline{\nu}\circ TB)^{-2}(\mathcal{DCP}^{(k,0,(k-1))}(\overline{\nu},T,B))\right)\nonumber \\
&+&\displaystyle \mathcal{DICP}^{(k,2,k)}(\overline{\nu},T,B).
\end{eqnarray}
Therefore, by replacing \eqref{I1.1F} into \eqref{U1k}, and on account of $Sym^{k} \circ Sym^{k}=Sym^{k},$ it follows formula \eqref{1k.F}.
\end{proof}
\begin{lema}\label{U2}
Under Definitions \ref{A,B,C} and \ref{DAL}. Assume that $\overline{\nu} \in \mathcal{A}_{L}$ is a $C^{k}$ function and that  $U_{2}^{k}(\overline{\nu},T,A,B,C):D^{*} \to L^{k}(\R^{n+1},\R^{n})$ is as in Lemma $\ref{DGamma}.$
 %that is,
%  \begin{equation}\label{DU2}
%U_{2}^{k}(\overline{\nu},T,A,B,C)(x,y)=D^{k}(\overline{\nu}\circ TA-C)(1-\overline{\nu}\circ B)^{-1}.
%  \end{equation}
Then, the following formulas hold:
\begin{equation}\label{U21F}
U_{2}^{1}(\overline{\nu},T,A,B,C)=\left((D\overline{\nu})\circ TDTA+\overline{\nu}\circ TDA-D(C)\right)\left(1-\overline{\nu}\circ TB\right)^{-1},
\end{equation}
for $k\geq 2$
\begin{eqnarray}\label{U2kF}
U_{2}^{k}(\overline{\nu},T,A,B,C)&=&\displaystyle Sym^{k}\left(\mathcal{DC}^{(k,k,k)}(\overline{\nu},T)A\right)(1-\overline{\nu}\circ TB)^{-1}\nonumber\\
&+&Sym^{k}\left(\mathcal{DC}^{(k,1,k-1)}(\overline{\nu},T)A\right)(1-\overline{\nu}\circ T B)^{-2}\nonumber \\
&+&Sym^{k}\left(\mathcal{DCP}^{(k,0,k-1)}(\overline{\nu},T,A)\right)(1-\overline{\nu}\circ T B)^{-2} \nonumber \\
&-&Sym^{k}\left(D^{k}(C)\right).(1-\overline{\nu}\circ T B)^{-2},
\end{eqnarray}
where $\mathcal{DC}^{(k_{1},k_{2},k_{3})}(\overline{\nu},T),$ $\mathcal{DCP}^{(k_{1},k_{2},k_{3})}(\overline{\nu},T,B),$ $\mathcal{DICP}^{(k_{1},k_{2},k_{3})}(\overline{\nu},T,B))$ as in Definitions \ref{DC}, \ref{DCP2} and \ref{DCPI1},  respectively.
\end{lema}
\begin{proof}
The proof is quite similar to the development from Eq. \eqref{I1.1}.
\end{proof}
\begin{lema}\label{U3}
Under Definitions \ref{A,B,C} and \ref{DAL}. Assume that $\overline{\nu} \in \mathcal{A}_{L}$ is a $C^{k}$ function and that   $U_{3}^{k}(\overline{\nu},T,A,B,C):D^{*} \to L^{k}(\R^{n+1},\R^{n})$ as in Lemma \ref{DGamma}.
%that is,
% \begin{equation}
%U_{3}^{k}(\overline{\nu},T,A,B,C)(x,y)=\sum_{q=1}^{k-1}{k \choose q}D^{q}(\overline{\nu}\circ TA-C)D^{k-q}(1-\overline{\nu}\circ TB)^{-1}(x,y).
%\end{equation}
Then the following equality holds:
\begin{multline}
U_{3}^{k}(\overline{\nu},T,A,B,C)=\\Sym^{k}\left(\sum_{q=1}^{k-1}{k \choose q}\phi^{(q,k-q)}((\mathcal{DCP}^{(q,0,q)}(\overline{\nu},T,A)-D^{q}(C)),\mathcal{DICP}^{(k-q,1,k-q)}(\overline{\nu},T,B))\right),
\end{multline}
where $\mathcal{DCP}^{(k_{1},k_{2},k_{3})}(\overline{\nu},T,B),$ $\phi^{(q,k-q)},$ $\mathcal{DICP}^{(k_{1},k_{2},k_{3})}(\overline{\nu},T,B))$ as in Definitions \ref{DCP2}, \ref{MM1} and  \ref{DCPI1}, respectively.
\end{lema}
\begin{proof}
The proof is similar to that of Lemma \ref{U1}. For more details, see all the developments of the formulas  $I_{1}^{k}$ (Eq. \eqref{I1.1}) and $I_{2}^{k}$ (Eq. \eqref{I2.1}).
\end{proof}
As a direct consequence of Lemmas \ref{DGamma}, \ref{U1}, \ref{U2} and \ref{U3} we obtain a formula for the $kth$ derivative of the function $\Gamma(\overline{\nu})$ at the point $(x,y),$ for $y\neq 0.$
\begin{lema}\label{MC}
Under Definitions \ref{A,B,C}, \ref{DAL} and $\ref{gamma}.$ Assume that $\overline{\nu} \in \mathcal{A}_{L}$ is a  $C^{k}$ function a that  $y \neq 0,$ then the following formulas hold:
\begin{multline}
D(\Gamma(\overline{\nu}))(x,y)=\left((\overline{\nu}\circ TA-C)(1-\overline{\nu} \circ T B)^{-2}\left(\overline{\nu}\circ TDB+D\overline{\nu}\circ T.DT.B\right)\right.\\
+\left.\left(\overline{\nu}\circ TDTA+\overline{\nu}\circ TDA-DC\right)\left(1-\overline{\nu}\circ T B\right)^{-1}\right)(x,y).\\
\end{multline}
 for $k\geq 2$
\begin{eqnarray}
D^{k}(\Gamma(\overline{\nu}))(x,y)&=&\left((\overline{\nu}\circ TA-C)(1-\overline{\nu}\circ TB)^{-2}Sym^{k}\circ \left(\mathcal{DC}^{(k,k,k)}(\overline{\nu},T)B\right)\right. \nonumber \\
&+&(\overline{\nu}\circ TA-C)(1-\overline{\nu}\circ TB)^{-2}Sym^{k}\circ\left(\mathcal{DC}^{(k,1,(k-1))}(\overline{\nu},T)\right)\nonumber\\
&+&(\overline{\nu}\circ TA-C)(1-\overline{\nu}\circ TB)^{-2}Sym^{k}\circ\left(\mathcal{DCP}^{(k,0,(k-1))}(\overline{\nu},T,B)\right)\nonumber\\
&+&\displaystyle (\overline{\nu}\circ TA-C)\mathcal{DICP}^{(k,2,k)}(\overline{\nu},T,B)\nonumber \\
&+&\displaystyle (1-\overline{\nu}\circ T B)^{-1} Sym^{k}\circ \left(\mathcal{DC}^{(k,k,k)}(\overline{\nu},T)A\right)\nonumber\\
&+&(1-\overline{\nu}\circ TB)^{-2}Sym^{k}\nonumber \left(\mathcal{DC}^{(k,1,k-1)}(\overline{\nu},T)A+\mathcal{DCP}^{(k,0,k-1)}(\overline{\nu},T,A)\right)\nonumber \\
&-& (1-\overline{\nu}\circ TB)^{-2}Sym^{k}\circ \left(D^{k}(C)\right)+U_{3}^{k}(\overline{\nu},T,A,B,C).
\end{eqnarray}
where $\mathcal{DC}^{(k_{1},k_{2},k_{3})}(\overline{\nu},T),$ $\mathcal{DCP}^{(k_{1},k_{2},k_{3})}(\overline{\nu},T,B),$ $\mathcal{DICP}^{(k_{1},k_{2},k_{3})}(\overline{\nu},T,B)),$ $\phi^{(q,k-q)}$ as in Definitions \ref{DC}, \ref{DCP2}, \ref{DCPI1} and \ref{MM1}, respectively and $U_{3}^{k}(\overline{\nu},T,A,B,C)$  as in Lemma $\ref{U3}.$
\end{lema}
\subsection{Part 2: The norm of the $ith$ derivative}
In this sub-section will be estimated the norms of the $ith$ derivative of the functions $A(x,y)$, $B(x,y)$ and $C(x,y)$ around of a neighborhood of $D_{0}.$ We start by noticing the following simple but useful lemma.
\begin{lema}\label{ld}
 Let
 \begin{equation}\label{Ed}
 \displaystyle d(x,y)=\begin{cases}
\displaystyle\alpha(A^{*}_{+}+\partial_{x}\psi_{+}(x,y))+|y|\partial_{y}\psi_{+}(x,y),& y> 0,\\\\
\displaystyle\alpha(A^{*}_{-}+\partial_{x}\psi_{-}(x,y))+|y|\partial_{y}\psi_{-}(x,y), & y< 0,
\end{cases}
\end{equation}
and
 \begin{equation}\label{rho}
 \displaystyle \rho(x,y)=\begin{cases}
\displaystyle \frac{1}{\displaystyle\min_{0\leq j \leq i}\mid(\alpha A^{*}_{+}+\psi_{+}(x,y)+y\partial_{y}\psi_{+}(x,y))\mid^{j+1}},& y> 0,\\\\
\displaystyle \frac{1}{\displaystyle\min_{0\leq j \leq i}\mid(\alpha A^{*}_{-}+\psi_{-}(x,y)+|y|\partial_{y}\psi_{-}(x,y))\mid^{j+1}}, & y< 0,
\end{cases}
\end{equation}
 Then,  $d$ and $\rho$ are defined in a neighborhood $\widetilde{U}$ of $D_{0}$ and there exists a constant $C\geq0$ such that the following estimative holds:
 \begin{equation}\label{kd}
 \norm{ D^{i}(d(x,y)^{-1})} \leq C \rho(x,y)|y|^{\gamma-i}, \quad \mbox{for all} \quad (x,y) \in \widetilde{U}.
 \end{equation}
 \end{lema}
 Moreover, the limit
\begin{equation}\label{lim:1}
\displaystyle \lim_{(x,y)\to (a,0^{\pm})}\rho(x,y)
\end{equation}
exists, for all $(a,0)\in D_{0}.$
 \begin{proof}
 Since $A^{*}_{\pm}\neq 0,$ then the estimate \eqref{kd}  is a directly consequence of Example \ref{DCPI4} and norm properties. From Assumption  \ref{Assumptions}$(L1),$ it follows that the limit $\lim_{(x,y)\to (a,0^{\pm})}\rho(x,y)$ exists. This finishes the proof of lemma.
 \end{proof}
 As a consequence of Lemma \ref{ld} and Leibnitz rule we get:
\begin{corolario}\label{DAC}
 Let $0 \leq i \leq k$ be a integer. Assume  $A, B$ and  $C$  as in Definition \ref{A,B,C} and $\rho$ as in Lemma \ref{ld}. Then, there is a constant $C\geq 0$ such that the following inequalities hold:
\begin{equation}
\parallel D^{i}A(x,y)\parallel \leq C \rho(x,y)|y|^{\gamma-i+1},
\end{equation}
\begin{equation}\label{DCk}
\parallel D^{i}C(x,y)\parallel \leq C \rho(x,y)|y|^{\gamma-i+1},
\end{equation}
\begin{equation}\label{DBk}
 ||D^{i}B(x,y)||\leq C \rho(x,y)|y|^{\gamma-i},
\end{equation}
in a neighborhood of $D_{0}.$
\end{corolario}
\begin{corolario}\label{NC1}
Assume $T(x,y)=(F(x,y),G(x,y))$ is a map that satisfies Assumption \ref{Assumptions}($L_{1})$. Then, the following relation holds:
\begin{equation}
\norm{D^{k}T(a,b)}\leq const \norm{b}^{\alpha-k},
\end{equation}
in a neighborhood of $D_{0},$ where $\textbf{const}$  denotes a positive constant.
\end{corolario}
\begin{proof}
The proof is a direct consequence of Assumption \ref{Assumptions}(Eq. \eqref{L1}) and Leibnitz rule.
\end{proof}
\begin{lema}\label{CU123}
Let $A, B$ and $C$ be as in Definition \ref{A,B,C}. Assume that $\overline{\nu}\in \mathcal{A}_{L}$ is a $C^{k}$ functions and that $U_{1}^{k}(\overline{\nu},T,A,B,C),$ $U_{2}^{k}(\overline{\nu},T,A,B,C)$ and $U_{3}^{k}(\overline{\nu},T,A,B,C)$  as in Lemmas \ref{U1}, \ref{U2} and \ref{U3}, respectively . Then
 \begin{equation}\label{CU1}
 \displaystyle \lim_{(a,b)\to (x,0)}U_{1}^{k}(\overline{\nu},T,A,B,C)(a,b)=0,
 \end{equation}
  \begin{equation}\label{CU2}
 \displaystyle \lim_{(a,b)\to (x,0)}U_{2}^{k}(\overline{\nu},T,A,B,C)(a,b)=0,
 \end{equation}
  \begin{equation}\label{CU3}
 \displaystyle \lim_{(a,b)\to (x,0)}U_{3}^{k}(\overline{\nu},T,A,B,C)(a,b)=0,
 \end{equation}
for every $(x,0)\in D_{0}.$
\end{lema}
\begin{proof}
The result is easy to prove for $k=1$. We prove the result for the case $k\geq 2.$

By Lemma \ref{U1}($k \geq 2$), Definition \ref{DAL} and  Remark \ref{RP}(iii) we have
\begin{eqnarray}\label{U1Ck2}
\norm{U^{k}_{1}(\overline{\nu},T,A,B,C)}&\leq& \displaystyle\frac{(L\norm{A(a,b)}+\norm{C(a,b)})}{(1-L\norm{B})^{2}}\norm{k!\mathcal{DC}^{(k,k,k)}(\overline{\nu},T)(a,b)}\norm{B}\nonumber\\
&+&\frac{(L\norm{A(a,b)}+\norm{C(a,b)})}{(1-L\norm{B})^{2}}\norm{k!\mathcal{DC}^{(k,1,(k-1))}(\overline{\nu},T)(a,b)}\nonumber\\ &+&\frac{(L\norm{A(a,b)}+\norm{C(a,b)})}{(1-L\norm{B})^{2}}\norm{\mathcal{DCP}^{(k,0,(k-1))}(\overline{\nu},T,B)(a,b)}\nonumber \\
&+&\displaystyle (L\norm{A(a,b)}+\norm{C(a,b)})\norm{\mathcal{DICP}^{(k,2,k)}(\overline{\nu},T,B)(a,b)}. \nonumber\\
\end{eqnarray}
To estimate the first expression of \eqref{U1Ck2}. From  \eqref{Dkkk} and norm properties we have
\begin{equation}\label{U1Ck4}
\norm{\mathcal{DC}^{(k,k,k)}(\overline{\nu},T)(a,b)}\leq  \norm{k!(D^{k}\overline{\nu}) \circ T}\norm{DT(a,b)}^{k}.
\end{equation}
Since $\nu$ is of class $C^{k},$ and by using Corollary \ref{NC1} we get
 \begin{equation}\label{U1Ck7}
\norm{\mathcal{DC}^{(k,k,k)}(\overline{\nu},T)(a,b)}\leq const|b|^{\gamma-k}.
\end{equation}
Whence, in view of  Corollary \ref{DAC} we get
\begin{equation}\label{U1Ck71}
\displaystyle\frac{(L\norm{A(a,b)}+\norm{C(a,b)})}{(1-L\norm{B})^{2}}\norm{k!\mathcal{DC}^{(k,k,k)}(\overline{\nu},T)(a,b)} \leq const \frac{|b|^{\alpha-k+\gamma+1}}{(1-L\norm{B})^{2}}.
\end{equation}
By similar arguments one can  estimates remaining expressions of \eqref{U1Ck2} to obtain
\begin{equation}\label{U1Ck101}
\frac{(L\norm{A(a,b)}+\norm{C(a,b)})}{(1-L\norm{B})^{2}}\norm{k!\mathcal{DC}^{(k,1,(k-1))}(\overline{\nu},T)(a,b)}\leq |b|^{\alpha-k+ \gamma+1}(1-L\norm{B})^{2}const.
\end{equation}
\begin{equation}\label{U1Ck131}
\frac{(L\norm{A(a,b)}+\norm{C(a,b)})}{(1-L\norm{B}_{D})^{2}}\norm{\mathcal{DCP}^{(k,0,(k-1))}(\overline{\nu},T,B)(a,b)}\leq \frac{|b|^{\gamma-k+1}}{(1-L\norm{B}_{D})^{2}}.
\end{equation}
\begin{equation}\label{U1Ck171}
\displaystyle (L\norm{A(a,b)}+\norm{C(a,b)})\norm{\mathcal{DICP}^{(k,2,k)}(\overline{\nu},T,B)(a,b)}\leq const |b|^{\alpha-k +\gamma+1}.
\end{equation}

Therefore, combining the four estimates \eqref{U1Ck171}, \eqref{U1Ck131}, \eqref{U1Ck101} and \eqref{U1Ck71} with \eqref{U1Ck2} we obtain
\begin{equation}\label{U1Ck21}
\norm{U^{k}_{1}(\overline{\nu},T,A,B,C)}\leq const|b|^{\alpha-k +\gamma+1}.
\end{equation}
Hence, since $\gamma>k-1$ and $\alpha>0$ (see Assumption $\ref{Assumptions}(L_{1}$)) we reach that
\begin{equation*}
\displaystyle \lim_{(a,b)\to (x,0)}\norm{U^{k}_{1}(\overline{\nu},T,A,B,C)(a,b)}=0.
\end{equation*}
Repeating  the same procedure followed to deduce estimate \eqref{CU1} , we get estimates \eqref{CU2} and \eqref{CU3}. Thus, we conclude the proof of corollary.
\end{proof}
\begin{proof}[\textbf{Proof of Proposition \ref{Di1}}]
This is a direct consequence of Lemma  \ref{CU123}.
\end{proof}
\section{Proof of Proposition \ref{propstep2}}\label{sec:step2}
The proof of Proposition \ref{propstep2} was influenced by the ideas contained in the articles \cite[p. 313]{Robinson} and \cite[Eq. (3)]{Shashkov}. The proof is quite long and technical, so we divide it into two steps. Before that, we give  the following definition.
\begin{definicao}\label{Dis2}
We define the set $D_{i}$ of all the continuous functions $\nu_{i}: D \to L^{i}_{s}(\R^{n+1},\R^{n})$ such that $\nu_{i}(x,0)=0,$ for all $(x,0) \in D_{0}$ that is,
\begin{equation*}
\mathcal{D}_{i}:=\{\nu_{i}: D \to L^{i}_{s}(\R^{n+1},\R^{n}):\nu_{i}(x,0)=0, \quad \mbox{for all} \quad (x,0) \in D_{0}; \quad
\mbox{$v_{i}$ is continuous}\},
\end{equation*}
for every $1\leq i \leq k,$ and $\mathcal{D}_{0}:=\mathcal{A}_{L}.$
\end{definicao}
The proof of Proposition  \ref{propstep2} is somewhat lengthy, so we divide it into two parts. \textbf{In the first part:} we show the
\emph{existence of functions $\Psi^{i}:\mathcal{D}_{0}\times \mathcal{D}_{1}\times \ldots \times \mathcal{D}_{i}\to D_{i},$ }so that $D^{i}(\Gamma(\overline{\nu}_{0}))=\Psi^{i}(\overline{\nu}_{0},D(\overline{\nu}_{0}),\ldots,D^{i}(\overline{\nu}_{0})),$ for all $0\leq i \leq k.$ \textbf{In the second part:} we show that  the function $\widetilde{N}_{i}:\mathcal{D}_{0}\times \mathcal{D}_{1}\times \ldots \times \mathcal{D}_{i} \to \mathcal{D}_{0}\times \mathcal{D}_{1}\times \ldots \times \mathcal{D}_{i} $ given by $\widetilde{N}_{i}(\overline{\nu}_{0},\overline{\nu}_{1},\ldots,\overline{\nu}_{i})=(\Gamma(\overline{\nu}_{0}),\Psi^{1}(\overline{\nu}_{0},\overline{\nu}_{1}),\ldots,\Psi^{i}(\overline{\nu}_{0},\overline{\nu}_{1},\ldots,\overline{\nu}_{i}))$
\emph{have a global attracting fixed point $(A_{0},A_{1},\ldots,A_{i}),$}
for all $0\leq i \leq k.$

\subsection{Part 1: Defining the functions $\Psi^{i}$} We start by defining a generalization of the function $U^{k}_{1}$ (see Corollary \ref{U1}).
\begin{definicao}\label{U1.N}
Let $1\leq i\leq k$ be a integer. Let $D_{0}$ and $\mathcal{D}_{i}$ be sets as in Definition \ref{Dis2}. We define the function
$ U^{i}_{1}:\mathcal{D}_{0}\times \mathcal{D}_{1}\times \ldots \times \mathcal{D}_{i}\to D_{i}$ given by
\begin{equation}\label{11.N}
 U_{1}^{1}(\overline{\nu}_{0},\overline{\nu}_{1})=(\overline{\nu}_{0}\circ TA-C)(1-\overline{\nu}_{0} \circ T B)^{-2}\left(\overline{\nu}_{0}\circ TDB+\overline{\nu}_{1}\circ T.DT.B\right),
 \end{equation}
 for $i\geq 2$
\begin{eqnarray}\label{U1i.N}
U^{i}_{1}(\overline{\nu}_{0},\ldots,\overline{\nu}_{i})&=&\frac{(\overline{\nu}_{0}\circ TA-C)i!}{(1-\overline{\nu}_{0}\circ TB)^{2}}Sym^{i}\circ \mathcal{DC}^{(i,i,i)}(\overline{\nu}_{i},T)B\nonumber\\
&+&\frac{(\overline{\nu}_{0}\circ TA-C)i!}{(1-\overline{\nu}_{0}\circ TB)^{2}}Sym^{i}\circ \mathcal{DC}^{(i,1,(i-1))}(\overline{\nu}_{1},\ldots,\overline{\nu}_{(i-1)},T)\nonumber\\
&+&\frac{(\overline{\nu}_{0}\circ TA-C)i!}{(1-\overline{\nu}_{0}\circ TB)^{2}}Sym^{i}\circ \mathcal{DCP}^{(i,0,(i-1))}(\overline{\nu}_{0},\ldots,\overline{\nu}_{(i-1)},T,B)
\nonumber \\
&+&\displaystyle (\overline{\nu}_{0}\circ TA-C)\mathcal{DICP}^{(i,2,i)}(\overline{\nu}_{0},\ldots,\overline{\nu}_{(i-1)},T,B),
\end{eqnarray}
where $\mathcal{DC}^{(k_{1},k_{2},k_{3})}(\overline{\nu},T),$ $\mathcal{DCP}^{(k_{1},k_{2},k_{3})}(\overline{\nu},T,B)$ and $\mathcal{DICP}^{(k_{1},k_{2},k_{3})}(\overline{\nu},T,B)$  as in Definitions \ref{DC}, \ref{DCP2} and \ref{DCPI1}, respectively.
\end{definicao}
Next, we define a generalization of the function $U^{k}_{2}$ (see Corollary \ref{U2}).
\begin{definicao}\label{U2.N}
Let $1\leq i\leq k$ be a integer. Let $D_{0},$ $\mathcal{D}_{i}$ be sets as in Definition \ref{Dis2}. We define the function
$ U^{i}_{2}:\mathcal{D}_{0}\times \mathcal{D}_{1}\times \ldots \times \mathcal{D}_{i}\to D_{i}$  given by
\begin{equation}\label{U21.N}
U_{2}^{1}(\nu_{0},\nu_{1},T)=(\nu_{1}\circ TDTA+\nu_{0}\circ TDA-DC)(1-\overline{\nu}_{0}\circ T B)^{-1},
\end{equation}
for $i\geq 2$
\begin{eqnarray}\label{U2i.N}
(U_{2}^{i})(\overline{\nu}_{0},\overline{\nu}_{1},\ldots,\overline{\nu}_{i})&=& \displaystyle (1-\overline{\nu}_{0}\circ T B)^{-1}Sym^{i}\circ \left(\mathcal{DC}^{(i,i,i)}(\overline{\nu}_{i},T)A\right)\nonumber\\&+&
(1-\overline{\nu}_{0}\circ T B)^{-2}Sym^{i}\circ \left(\mathcal{DC}^{(i,1,i-1)}(\overline{\nu}_{1}, \ldots, \overline{\nu}_{i-1},T)A-D^{i}(C)\right)\nonumber\\
&+&(1-\overline{\nu}_{0}\circ T B)^{-2}Sym^{i}\circ \left(\mathcal{DCP}^{(i,0,i-1)}(\overline{\nu}_{0},\overline{\nu}_{1},\ldots,\overline{\nu}_{(i-1)},T,A)\right)
\end{eqnarray}
where $\mathcal{DC}^{(k_{1},k_{2},k_{3})}(\overline{\nu},T)$ and $\mathcal{DCP}^{(k_{1},k_{2},k_{3})}(\overline{\nu},T,B),$ as in Definitions \ref{DC} and \ref{DCP2}, respectively.
\end{definicao}
Next, we define a generalization of the function $U^{k}_{3}$ (see Corollary \ref{U3}).
\begin{definicao}\label{U3.N}
Let $2\leq i\leq k$ be a integer. Let $D_{0},$ $\mathcal{D}_{i}$ be sets as in Definition \ref{Dis2}. We define the function
$U^{i}_{3}:\mathcal{D}_{0}\times \mathcal{D}_{1}\times \ldots \times \mathcal{D}_{i}\to D_{i}$ given by
\begin{multline}
(U_{3}^{i})(\overline{\nu}_{0},..,\overline{\nu}_{i})=\\
Sym^{i}\circ\sum_{q=1}^{i-1}{i \choose q}\phi^{(q,i-q)}(\mathcal{DCP}^{(q,0,q)}(\overline{\nu}_{0},..,\overline{\nu}_{q},T,A)-D^{i}C,\mathcal{DICP}^{(i-q,1,i-q)}(\overline{\nu}_{0},..,\overline{\nu}_{i-q-1},T,B)),
\end{multline}
where $\mathcal{DCP}^{(k_{1},k_{2},k_{3})}(\overline{\nu},T,B),$ $\mathcal{DICP}^{(k_{1},k_{2},k_{3})}(\overline{\nu},T,B)),$ $\phi^{(q,i-q)}$ as in Definitions  \ref{DCP2}, \ref{DCPI1} and \ref{MM1}, respectively.
\end{definicao}
Next, we define a generalization of the function $D^{k}\Gamma(\nu)$ (see Lemma \ref{MC}).
\begin{definicao}\label{psi}
Let $1\leq i\leq k$ be a integer. Let $D_{0},$ $\mathcal{D}_{i}$ be sets as in Definition \ref{Dis2}, and assume the functions $U_{1}^{i}, U_{2}^{i}$ and $U_{3}^{i}$ as in Definitions \ref{U1.N}, \ref{U2.N} and \ref{U3.N}, respectively. We define the function
$\Psi^{i}:\mathcal{D}_{0}\times \mathcal{D}_{1}\times \ldots \times \mathcal{D}_{i}\to D_{i}$ given by
\begin{equation}
\Psi^{1}(\overline{\nu}_{0},\overline{v}_{1}, \overline{\nu}_{2})=(U_{1}^{1}+U_{2}^{1})(\overline{\nu}_{0},\overline{\nu}_{1}),
\end{equation}
and
for $i\geq2$
\begin{equation}
\Psi^{i}(\overline{\nu}_{0},\overline{\nu}_{1},\ldots,\overline{\nu}_{i})=(U_{1}^{i}+U_{2}^{i}+U_{3}^{i})(\overline{\nu}_{0},\overline{\nu}_{1},\ldots,\overline{\nu}_{i}).
\end{equation}
\end{definicao}
\begin{observacao}
The functions $\Psi^{i}$ for the cases $i=1$ and $i=2$ were established in  \cite[Eq. (13)]{Shashkov} and  \cite[Eq. (42)]{Pacifico} respectively.
\end{observacao}
\begin{prop}\label{wdpsi}
Let $1\leq i\leq k$ be a integer. Then, the function $\Psi^{i}$ given in Definition \ref{psi} is well-defined. Moreover, if $\overline{\nu}_{0}\in \mathcal{A}_{L}$ is of class $C^{i},$ then
\begin{equation}\label{psi1}
\Psi^{i}(\overline{\nu}_{0},D\overline{\nu}_{0},\ldots,D^{i}\overline{\nu}_{0})=D^{i}\Gamma(\overline{\nu}_{0}).
\end{equation}
\end{prop}
\begin{proof}
To prove that the function  $\Psi^{i}$ is well-defined, it suffices to show that
\begin{equation}\label{psi.2}
\Psi^{i}(\overline{\nu}_{0},\ldots,\overline{\nu}_{i}) \in \mathcal{D}_{i}, \quad \mbox{for all} \quad   \overline{\nu}_{j} \in \mathcal{D}_{j}, 0\leq j\leq 1.
\end{equation}
That is, by Definition \ref{Dis2} we must to show that
\begin{itemize}\label{psi.3}
\item[(a)]$\Psi^{i}(\overline{\nu}_{0},\ldots,\overline{\nu}_{i})$ is continuous on $D$
and
\item[(b)]$\Psi^{i}(\overline{\nu}_{0},\ldots,\overline{\nu}_{i})(x,0)=0,$ for every $x\in \R^{n},$ $\norm{x}\leq 1,$
\end{itemize}
for all  $\overline{\nu}_{j} \in \mathcal{D}_{j}, 0\leq j\leq 1.$ Indeed, by Definition \ref{psi} we have that $\Psi^{i}(\overline{\nu}_{0},\ldots,\overline{\nu}_{i})$ is continuous on $D^{*},$ so it remains to show the continuity of $\Psi^{i}(\overline{\nu}_{0},\ldots,\overline{\nu}_{i})$ at the points $(x,0) \in D_{0}.$ Analysis similar to that in the proof of  Proposition \ref{Di1} shows that
\begin{equation}\label{psi(x,0)}
\displaystyle \lim_{(a,b)\to (x,0)}\psi^{i}(\overline{\nu}_{0},\ldots,\overline{\nu}_{i})(a,b)=0,
\end{equation}
for all $(x,0)\in D_{0}.$  Therefore, if we define
  \begin{equation*}
 W^{i}(x,y)=\begin{cases}
\Psi^{i}(\overline{\nu}_{0},\ldots,\overline{\nu}_{i})(x,y),& y\neq 0,\\
0, & y=0,
\end{cases}
\end{equation*}
then, we get a continuous extension of $\Psi^{i}(\overline{\nu}_{0},\ldots,\overline{\nu}_{i})$ on $D,$ which completes the proof of \eqref{psi.3}, so $\Psi^{i}(\overline{\nu}_{0},\ldots,\overline{\nu}_{i}) \in \mathcal{D}_{i}.$
Therefore $\Psi^{i}$ is well-defined. The equality in Eq. \eqref{psi1} follows from Definition \ref{psi} and Lemma \ref{MC}. This concludes the proof.
\end{proof}
\subsection{Part 2: The function $\widetilde{N}_{i}$}
In this sub-section will be shown the following proposition.
\begin{prop}\label{l3}
Let $1\leq i\leq k$ be a integer. Let $\Psi^{j}, 1 \leq j \leq i$ be functions as in Definition \ref{psi}. Then the function
\begin{equation}\label{Ni}
\widetilde{N}_{i}:\mathcal{D}_{0}\times \mathcal{D}_{1}\times \ldots \times \mathcal{D}_{i} \to \mathcal{D}_{0}\times \mathcal{D}_{1}\times \ldots \times \mathcal{D}_{i}
\end{equation}
given by
\begin{equation}
\widetilde{N}_{i}(\overline{\nu}_{0},\overline{\nu}_{1},\ldots,\overline{\nu}_{i})=(\Gamma(\overline{\nu}_{0}),\Psi^{1}(\overline{\nu}_{0},\overline{\nu}_{1}),\ldots,\Psi^{i}(\overline{\nu}_{0},\overline{\nu}_{1},\ldots,\overline{\nu}_{i})),
\end{equation}
have a global attracting fixed point $(A_{0},A_{1},\ldots,A_{i}).$
\end{prop}

\subsubsection{Preliminares}
Before proving Proposition \ref{l3}, we state without proof two theorem which will be useful in the sequel.
 \begin{teorema}[Fiber Contraction Theorem \cite{HP}]\label{FCT} Let $(X,d_{X})$ and $(Y,d_{Y})$ be two complete metric spaces, and let $\Upsilon:X\times Y \to X\times Y$ be a map of the form
$$\Upsilon(x,y)=(\Gamma(x),\Psi(x,y)).$$
Assume that
  \begin{itemize}
  \item[$(a)$]$\Gamma$ has an attracting fixed point $x_{\infty}$, that is,
  \begin{center}
  $\Gamma(x_{\infty})=x_{\infty},$ $\displaystyle\lim_{n \to \infty}\Gamma^{n}(x)=x_{\infty}, \quad \mbox{for all} \quad x\in X;$
  \end{center}
 \item[$(b)$] the family of functions $\Psi^{y}:X \to Y$ given by $\Psi^{y}(x)=\Psi(x,y)$ depends on $y$ continuously; that is, if $x_{n}\to x$ as $n \to \infty,$ then $\Psi^{y}(x_{n})\to \Psi^{y}(x)$ as $n \to \infty.$
  \item[$(c)$]for every $x\in X$ the map $\Psi_{x}:=\Psi(x,.):Y \to Y$ defined by $\Psi_{x}(y):=\Psi(x,y)$ is a $\lambda$-contraction, with $\lambda<1.$
      This mean that
      $$d_{Y}(\Psi_{x}(y_{1}),\Psi_{x}(y_{2}))\leq \lambda d_{Y}(y_{1},y_{2}),$$
      for all $x\in X$ and $y_{1}, y_{2} \in Y.$

  \end{itemize}
  Then, if $y_{\infty}$ denotes the unique fixed point of $\Psi_{x_{\infty}},$ the point $(x_{\infty}, y_{\infty}) \in X\times Y$ is an attracting fixed point of $\Upsilon$, that is,
    \begin{equation*}
      \displaystyle \lim_{n\to \infty}\Upsilon^{n}(x,y)=(x_{\infty}, y_{\infty}).
    \end{equation*}
  \end{teorema}

\begin{teorema}[Perron-Frobenius Theorem for positive matrices \cite{Meyer}]\label{PFT}
Let $A=[a_{i,j}]_{n \times n}$ be a real $n \times n$ positive matrix: $a_{i,j}>0,$ for $1\leq i,j\leq n.$ Then:
\begin{itemize}
\item[$(a)$] $A$ has a positive simple eigenvalue $r$ which is equal to the spectral radius of $A.$
\item[$(b)$] There exists an eigenvector $x$ with all the coordinates positives such that $Ax=rx.$
\item[$(c)$] The eigenvector is the unique vector defined by
\begin{center}
$Ap=rp,$ $p>0,$ and $\parallel p \parallel_{1}=1$, where $\parallel p \parallel_{1}=\sum_{i=1}^{n}|p_{i}|,$
\end{center}
and, except for positive multiples of $p,$ there are no other nonnegative eigenvector for $A,$ regardless of the eigenvalue.
   \item[$(d)$] An estimate of $r$ is given by inequalities:
   $$ \min_{i}\sum_{j}a_{ij} \leq r \leq \max_{i}\sum_{j}a_{ij}.$$

\end{itemize}
\end{teorema}

 We will now given some elementary properties of multilinear maps. Let us start by fixing the notion. The set $\{1, ..., n\}$ will be denoted by $[n].$ If  $E:=\R^{n}$ and $F:=\R,$  then $\mathcal{F}([k],\{E,F\})$  denoted the set  of all the functions $f:[k] \to \{E,F\}$. Notice that the cardinality of $\mathcal{F}([k],\{E,F\})$ is $2^{k}.$  Finally  $\pi_{_{E}}:\R^{n+1}\to \R^{n+1}$ and $\pi_{_{F}}:\R^{n+1}\to \R^{n+1}$  denoted the  projections of $\R^{n+1}$ on $E$ along $F$ and of $\R^{n+1}$ on $F$ along $E$ respectively.
\begin{definicao}\label{gfE}
Assume that $f \in \mathcal{F}([k],\{E,F\})$ and that  $\heartsuit=E$  or $\heartsuit=F.$ Then, define  \\ $\displaystyle g_{_{f,\heartsuit}}:[n+1]\to \{E,F\}$  by
\begin{equation*}
 g_{_{f,\heartsuit}}(i)=\begin{cases}
\displaystyle f(i),&\mbox{if}  \quad i \in [n],\\
\heartsuit, &\mbox{if}\quad i=n+1.
\end{cases}
\end{equation*}
By $\Omega([n+1],\{\heartsuit\})$ we denote the set of all functions $g_{_{f,\heartsuit}}.$
\end{definicao}
By $A \uplus B$ we denote the usual disjoint intersection between sets.
% \begin{equation}\label{FK}
% \mathcal{F}([k],\{E,F\}):\{f:[k]\to \{E,F\}:f \hspace{0.1cm}\mbox{is a function}\}.
% \end{equation}
 \begin{lema}\label{F(n+1)}The following statement holds:
 \begin{itemize}
 \item[(a)]$\Omega([n+1],\{E\}) \uplus \Omega([n+1],\{F\})=\mathcal{F}([n+1],\{E,F\}).$

 \end{itemize}
 \end{lema}
\begin{proof}
The proof follows immediately from Definition  \ref{gfE}.
\end{proof}
Recall that $L^{k}(\R^{n+1},\R^{n})$ denoted the space of all the $k$-linear maps from $\R^{n+1}$ to $\R^{n}.$
\begin{definicao}\label{Lf}
 Assume that $f \in \mathcal{F}([k],\{E,F\}).$ Then, the set of all $k$-linear maps $b$ such that
\begin{itemize}
\item [(a)]$b(\pi_{g(1)}(x_{1}),\pi_{g(2)}(x_{2}),\ldots \pi_{g(k)}(x_{k}))=0,$ for every $g \in  \mathcal{F}([k],\{E,F\}),$ $g\neq f$ and for each $k$-tuple $(x_{1}\ldots,x_{k})\in \underbrace{\R^{n+1}\times \ldots \times \R^{n+1}}_{k-times}.$
\end{itemize}
 will be denoted by
$$L^{k}_{f}(\R^{n+1},\R^{n})$$
\end{definicao}
\begin{lema}\label{LD0}
We have the following properties:
\begin{itemize}
\item[(a)] If $b \in L^{k}(\R^{n+1},\R^{n}),$ then
\begin{equation*}\label{LD}
b(x_{1},x_{2},\ldots,x_{k})=\sum_{f \in \mathcal{F}([k],\{E,F\})}b(\pi_{f(1)}(x_{1}),\ldots,\pi_{f(k)}(x_{k})),
\end{equation*}
for every $(x_{1},\ldots,x_{k})\in \underbrace{\R^{n+1}\times \ldots \times \R^{n+1}}_{k-times}.$ \\ The function $(x_{1},x_{2},\ldots,x_{k}) \to b(\pi_{f(1)}(x_{1}),\ldots,\pi_{f(k)}(x_{k}))$ will be denoted by $b_{f}.$
  \item[(b)] If $f \in \mathcal{F}([k],\{E,F\})$ and $b \in \displaystyle L^{k}_{f}(\R^{n+1},\R^{n}),$ then
 \begin{equation*}\label{LD.1}
b(x_{1},x_{2},\ldots,x_{k})=b(\pi_{f(1)}(x_{1}),\pi_{f(2)}(x_{2}),\ldots,\pi_{f(k)}(x_{k})),
\end{equation*}
 for every $(x_{1},\ldots,x_{k})\in \underbrace{\R^{n+1}\times \ldots \times \R^{n+1}}_{k-times}.$
  \item[(c)]$L^{k}(\R^{n+1},\R^{n})$ can be decomposed into a direct sum of $2^{k}$ $k$-linear maps that is,
 \begin{equation*}\label{LD.2}
% \nonumber to remove numbering (before each equation)
 L^{k}(\R^{n+1},\R^{n}) = \bigoplus_{f\in \mathcal{F}([k],\{E,F\})}\displaystyle L^{k}_{f}(\R^{n+1},\R^{n}),
\end{equation*}
where $L^{k}_{f}(\R^{n+1},\R^{n})$ as in Definition \ref{Lf}.
\end{itemize}
\end{lema}
\begin{proof}
The proof follows from Lemma \ref{F(n+1)} and Eq. \eqref{Lf}.
\end{proof}
 \subsubsection{Proof of Proposition \ref{l3}}

In order to prove Proposition \ref{l3} we state and prove the following proposition.
\begin{prop}\label{pfpsi0}
 Under the notation of Definitions \ref{Dis2} and $\ref{psi}.$ Let  $1\leq i \leq k$ be a integer and fix a point   $(\overline{\nu}_{0},\ldots,\overline{\nu}_{i-1}) \in \mathcal{D}_{0}\times \mathcal{D}_{1}\times \ldots \times \mathcal{D}_{i-1}.$ Then, the space $D_{i}$ can be endowed with a norm $|.|_{i,D}$ equivalent to the original norm $||.||_{D}$  such that the function $$\Psi^{i}(\overline{\nu}_{0},\ldots,\overline{\nu}_{i-1},\bullet):\mathcal{D}_{i} \to \mathcal{D}_{i}$$  is a contraction with  constant of contraction independent of the point $(\overline{\nu}_{0},\overline{\nu}_{1},\ldots,\overline{\nu}_{i-1}).$
\end{prop}
The proof of Proposition \ref{pfpsi0} will be given after some lemmas. We set,
\begin{equation}\label{DT:1}
\displaystyle \widehat{DT}(x,y):=\left[
                                                  \begin{array}{cc}
                                                    A(x,y) & B(x,y) \\
                                                    C(x,y) & 1\\
                                                  \end{array}
                                                \right]_{(n+1)\times(n+1)},
\end{equation}
where the functions $A(x,y), B(x,y)$ and $C(x,y)$ are as in Definition $\ref{A,B,C}.$
\begin{lema}\label{|Mib|}
Let $M^{i}:L^{i}(\R^{n+1},\R^{n}) \to L^{i}(\R^{n+1},\R^{n})$ be a map defined by
\begin{equation*}\label{|Mib.1|}
 M^{i}(b)(x_{1},\ldots,x_{i})=b(\widehat{DT}x_{1},\ldots, \widehat{DT}x_{i}).
 \end{equation*}
 Then, the space $L^{i}(\R^{n+1},\R^{n})$ can be endowed with a norm $|.|_{i}$ equivalent to $\norm{.}$ such that
\begin{equation}\label{Claim1}
\displaystyle \frac{|M^{i}(b)|_{i}}{\displaystyle |b|_{i}}\leq  \max_{ \stackrel{ m,n \in \N}{ m+n=i}}\{(||A||_{D}+||B||_{D})^{m}(||C||_{D}+1)^{n}\}.
\end{equation}
\end{lema}
\begin{proof}
Through of the proof, we deal with the case that $\norm{B}_{D}$ is  nonzero, the other case is similar. We will endow $L^{i}(\R^{n+1},\R^{n})$ with a new norm $|.|_{i}$ in the following way: letting
\begin{equation}\label{cfg0}
c_{g,f}:=||\displaystyle \pi_{g(1)}\widehat{DT}\pi_{f(1)}||\ldots ||\pi_{g(i)}\widehat{DT}\pi_{f(i)}||,
\end{equation}
 where $g$ and $f$ $\in\mathcal{F}([i],\{E,F\})$ while
\begin{equation}\label{Cfg}
\pi_{g(j)}\widehat{DT}\pi_{f(j)}:=\left\{\begin{array}{rc}
A,&\mbox{if}  \hspace{0.1cm}  g(j)=E \hspace{0.1cm}\mbox{and} \hspace{0.1 cm}f(j)=E,\\
B, &\mbox{if}\hspace{0.1cm}  g(j)=E \hspace{0.1cm}\mbox{and} \hspace{0.1 cm} f(j)=F,\\
C,&\mbox{if} \hspace{0.1cm}  g(j)=F \hspace{0.1cm}\mbox{and} \hspace{0.1 cm}  f(j)=E,\\
1, &\mbox{if}\hspace{0.1cm}  g(j)=F \hspace{0.1cm}\mbox{and} \hspace{0.1 cm}  f(j)=F.
\end{array}\right.
\end{equation}
Next up, consider the matrix
\begin{equation}\label{DELTA}
\Delta:= [c_{g,f}] _{2^{i}\times 2^{i}}.
 \end{equation}
Notice that since, by assumption $||A||_{D},||B||_{D}$ and $||C||_{D}$ are nonzero, then, in view of \eqref{cfg0} and \eqref{Cfg} it follows that $c_{g,f}>0,$ where $g,f  \in\mathcal{F}([i],\{E,F\}).$ Thus, the matrix $\Delta$ is positive. Therefore, by  Perron-Frobenius  Theorem \ref{PFT} applied to matrix $\Delta,$ we get
  \begin{itemize}\label{DELTAP}
  \item[(a)]The matrix $\Delta$ has a positive eigenvalue $\lambda$.
  \item[(b)]The matrix $\Delta$ has an eigenvector $V$ with entries $k_{f}$ such that
  \begin{equation*}\label{DELTAPb}
  \displaystyle\sum_{f \in\mathcal{F}([i],\{E,F\})}k_{f}=1.
  \end{equation*}
  \item[(c)]An  estimate of $\lambda$ is given by inequalities
  \begin{equation}\label{DELTAPc}
   \min_{g}\sum_{f}c_{g,f} \leq  \lambda\leq \displaystyle \max_{g}\sum_{f}c_{g,f}.
  \end{equation}
  \end{itemize}
Let $b \in L^{i}(\R^{n+1},\R^{n}), b\neq0$. In view of  Lemma \ref{LD0}(a) we can write
\begin{equation}\label{b}
b=\displaystyle\sum_{f \in \mathcal{F}([i],\{E,F\})}b_{f},
\end{equation}
where $b_{f}$ is as in as in Definition \ref{Lf}. Thus, we can define the norm $|.|_{i}$ on $L^{i}(\R^{n+1},\R^{n})$ by
\begin{equation}\label{|b|}
|b|_{i}:=\sum_{f \in \mathcal{F}([i],\{E,F\})}k_{f}\norm{b_{f}}.
\end{equation}
It is easily to seen that $|.|_{i}$ is a norm on $L^{i}(\R^{n+1},\R^{n})$  equivalent to the norm $||.||.$ \\
We now will prove that
\begin{equation}\label{||i}
\displaystyle \frac{|M^{i}(b)|_{i}}{\displaystyle |b|_{i}}\leq  \max_{ \stackrel{ m,n \in \N}{ m+n=i}}\{(||A||_{D}+||B||_{D})^{m}(||C||_{D}+1)^{n}\}.
\end{equation}
Indeed, by definition one has $M^{i}(b)$ is $i$-linear map, then on account of Lemma \ref{LD0}(c) and \eqref{|b|} we have
\begin{equation}\label{|Mb|}
|M^{i}(b)|_{i}:=\sum_{f \in \mathcal{F}([i],\{E,F\})}k_{f}\norm{M^{i}(b)_{f}},
\end{equation}
where $M^{i}(b)_{f}$ as in Definition \ref{Lf}. But, by using Lemma \ref{LD0}(b) we have
\begin{equation}\label{||i1}
M^{i}(b)_{f}(x_{1},\ldots,x_{i})=M^{i}(b)(\pi_{f(1)}(x_{1}),\ldots,\pi_{f(i)}(x_{i}))
\end{equation}
and by assumption \eqref{|Mib.1|} we have
\begin{equation}\label{||i2}
 M^{i}(b)(x_{1},\ldots,x_{i})=b(\widehat{DT}x_{1},\ldots, \widehat{DT}x_{i}).
\end{equation}
Thus, combining  \eqref{||i2} and \eqref{||i1} we get
\begin{equation}\label{||i3}
M^{i}(b)_{f}(x_{1},\ldots,x_{i})=b(\widehat{DT}\pi_{f(1)}x_{1},\ldots, \widehat{DT}\pi_{f(i)}x_{i}).
\end{equation}
Furthermore, by using Lemma \ref{LD0}(a) we can write
\begin{equation}\label{||i4}
b(\widehat{DT}\pi_{f(1)},\ldots,\widehat{DT}\pi_{f(i)})=\displaystyle \sum_{g \in \mathcal{F}([i],\{E,F\})}b_{g}(\pi_{g(1)}\widehat{DT}\pi_{f(1)},\ldots,\pi_{g(i)}\widehat{DT}\pi_{f(i)}).
\end{equation}
Therefore, it follows from \eqref{||i4} and \eqref{||i3}, that
\begin{equation}\label{||i5}
M^{i}(b)_{f}(x_{1},\ldots,x_{i})=\displaystyle \sum_{g \in \mathcal{F}([i],\{E,F\})}b_{g}(\pi_{g(1)}\widehat{DT}\pi_{f(1)}(x_{1}),\ldots,\pi_{g(i)}\widehat{DT}\pi_{f(i)}(x_{i})).
\end{equation}
Hence, on account of  \eqref{|Mb|}  we get
\begin{equation}\label{|Mib1|}
|M^{i}(b)|_{i}\leq\displaystyle \sum_{f \in \mathcal{F}([i],\{E,F\})}k_{f}\sum_{g \in \mathcal{F}([i],\{E,F\})}\norm{b_{g}(\pi_{g(1)}\widehat{DT}\pi_{f(1)},\ldots,\pi_{g(i)}\widehat{DT}\pi_{f(i)})}.
\end{equation}
Since $b_{g}$ is $i$-linear map, we have
\begin{equation*}\label{|Mib1.1|}
||\displaystyle b_{g}(\pi_{g(1)}\widehat{DT}\pi_{f(1)},\ldots,\pi_{g(i)}\widehat{DT}\pi_{f(i)})||\leq \underbrace{||\displaystyle \pi_{g(1)}\widehat{DT}\pi_{f(1)}||\ldots ||\pi_{g(i)}\widehat{DT}\pi_{f(i)}||}_{:=c_{g,f}}.\displaystyle ||b_{g}||.
\end{equation*}
Consequently, Eq. $\eqref{|Mib1|}$ becomes
\begin{equation}\label{|Mib2|}
|M^{i}(b)|_{i}\leq \displaystyle\sum_{g \in \mathcal{F}([i],\{E,F\})}||b_{g}||\sum_{f \in \mathcal{F}([i],\{E,F\})}c_{g,f} k_{f}.
\end{equation}
Notice that, since $V=[k_{f}]_{_{f\in \mathcal{F}([i],\{E,F\})}}$ is an eigenvector of  matrix $\Delta,$ we have
\begin{equation}\label{|Mib2.0|}
\Delta[k_{f}]_{_{f\in \mathcal{F}([i],\{E,F\})}} =\lambda[k_{f}]_{_{f\in \mathcal{F}([i],\{E,F\})}},
\end{equation}
where $\Delta=[c_{g,f}]$ while $g$ and $f \in \mathcal{F}([i],\{E,F\}).$\\
Hence, if we fix $g \in \mathcal{F}([i],\{E,F\})$ it is easily seen that
\begin{equation}\label{|Mib2.1|}
\sum_{f \in \mathcal{F}([i],\{E,F\})}c_{g,f}k_{f}=\lambda k_{g}.
\end{equation}
Thus, by replacing \eqref{|Mib2.1|} into  \eqref{|Mib2|} we get
\begin{equation}\label{|Mib3|}
|M^{i}(b)|_{i}\leq \displaystyle \lambda \sum_{g \in \mathcal{F}([i],\{E,F\})}||b_{g}||k_{g}.
\end{equation}
Moreover, by definition we can write
 \begin{equation}\label{|Mib3.1|}
|b|_{i}=\sum_{g \in \mathcal{F}([i],\{E,F\})}||b_{g}||k_{g}.
\end{equation}
Therefore, from \eqref{|Mib3.1|} and \eqref{|Mib3|}  one obtains
\begin{equation}\label{|Mib4|}
|M^{i}(b)|_{i}\leq \lambda |b|_{i}.
\end{equation}

Through of the remainder of the proof, we denote by $\#(S)$ the cardinality of the set $S$.
\begin{af}\label{CScgf}
  Let $f$ and $g \in \mathcal{F}([i],\{E,F\} )$  such that $\#(g^{-1}(E))=m$ and $\#(g^{-1}(F))=n.$ Then the following equality holds:
\begin{equation}\label{Scgf}
\displaystyle \sum_{f \in \mathcal{F}([i],\{E,F\}}c_{g,f}=(||A||_{D}+||B||_{D})^{m}(||C||_{D}+1)^{n},
\end{equation}
where $c_{g,f}$ as in Eq. \eqref{cfg0}.
%$\lambda \leq \max_{ \stackrel{ m,n \in \N}{ m+n=i}}\{(||A||_{D}+||B||_{D})^{m}(||C||_{D}+1)^{n}\}.$
\end{af}
Proof of the Claim. Since $\#(g^{-1}(E))=m$ and $\#(g^{-1}(F))=n,$ then one can consider\\
$g^{-1}(E):=\{a_{1},a_{2},\ldots,a_{m}\}$ and $g^{-1}(F):=\{b_{1},b_{2},\ldots,b_{n}\}.$
Thus, by definition we have
\begin{equation}\label{Cfg1}
\pi_{g(a_{i})}\widehat{DT}\pi_{f(a_{i})}:=\left\{\begin{array}{rc}
A,&\mbox{if}  \hspace{0.1cm} \hspace{0.1 cm} f(a_{i})=E,\\
B, &\mbox{if}\hspace{0.1cm}  \hspace{0.1cm} f(a_{i})=F,
\end{array}\right.
\end{equation}
and
\begin{equation}\label{Cfg2}
\pi_{g(b_{i})}\widehat{DT}\pi_{f(b_{i})}:=\left\{\begin{array}{rc}
C,&\mbox{if}  \hspace{0.1cm} \hspace{0.1 cm} f(b_{i})=E,\\
1, &\mbox{if}\hspace{0.1cm}  \hspace{0.1cm} f(b_{i})=F.
\end{array}\right.
\end{equation}
Now, consider integers $s,t$ with $0 \leq s \leq m,$ $0\leq t \leq n$ and take $f\in \mathcal{F}([i],\{E,F\})$ such that
\begin{equation*}
\#(g^{-}(E) \cap f^{-1}(E))=s \quad \mbox{and} \quad\#(g^{-}(F)\cap f^{-1}(E))=t.
\end{equation*}
Then, since $c_{g,f}:=||\displaystyle \pi_{g(1)}\widehat{DT}\pi_{f(1)}||\ldots ||\pi_{g(i)}\widehat{DT}\pi_{f(i)}||,$ it follows  from \eqref{Cfg2} and \eqref{Cfg1} that
\begin{equation}\label{cfg4}
c_{g,f}=||A||_{D}^{s}||B||_{D}^{m-s}||C||_{D}^{t}1^{n-t}.
\end{equation}
In addition, since  $\#{g^{-1}(E)}=m$ and $\#{g^{-1}(F)}=n,$ it is not difficult to see that the cardinality of the sets
\begin{equation}\label{Cards}
\mathcal{F}_{g,s}([i],\{E,F\}):=\{f\in \mathcal{F}([i],\{E,F\}):card(g^{-}(E) \cap f^{-1}(E))=s \}
\end{equation}
and
\begin{equation}\label{Cardt}
\mathcal{F}_{g,t}([i],\{E,F\}):=\{f\in \mathcal{F}([i],\{E,F\}):card(g^{-}(E) \cap f^{-1}(E))=t \}
\end{equation}
are  ${n \choose t}$ and  ${m \choose s},$ respectively. Thus, from \eqref{Cardt} and \eqref{Cards}, on account of Rule of Product ~\cite[p. 13]{Cohen} we deduce that the cardinality of
\begin{equation}\label{Cardst}
\mathcal{F}_{g,st}([i],\{E,F\}):= \mathcal{F}_{g,s}([i],\{E,F\})\cap \mathcal{F}_{g,t}([i],\{E,F\})
\end{equation}
is ${m \choose s}.{n \choose t}.$ Finally, notice that
\begin{equation*}\label{Cardst1}
\mathcal{F}([i],\{E,F\})=\displaystyle \biguplus_{\stackrel{0\leq s\leq m}{0\leq t\leq n}} \mathcal{F}_{g,st}([i],\{E,F\}).
\end{equation*}
Whence, on account of \eqref{cfg4} and Binomial Theorem we get the following chain of equalities
\begin{eqnarray*}\label{Cardst2}
\displaystyle \sum_{f \in\mathcal{F}([i],\{E,F\})}c_{g,f}&=&\displaystyle \sum_{s=0}^{m}\sum_{t=0}^{n}\displaystyle \left(\sum_{f \in \mathcal{F}_{g,st}([i],\{E,F\})}c_{g,f}\right) \nonumber\\ \nonumber\\
&=&\displaystyle \sum_{s=0}^{m}\sum_{t=0}^{n}\displaystyle {m \choose s}.{n \choose t}||A||_{D}^{s}||B||_{D}^{m-s}||C||_{D}^{t}1^{n-t}\nonumber\\ \nonumber \\
&=&\displaystyle (||A||_{D}+||B||_{D})^{m}(||C||_{D}+1)^{n}.
\end{eqnarray*}
Thus Claim \ref{CScgf} is proved.

Finally, from \eqref{DELTAPc} and  Claim \ref{CScgf} we conclude that
 \begin{equation}\label{Claim2}
\displaystyle \frac{|M^{i}(b)|_{i}}{\displaystyle |b|_{i}}\leq  \max_{\stackrel{ m,n \in \N}{ m+n=i}}\{(||A||_{D}+||B||_{D})^{m}(||C||_{D}+1)^{n}\},
\end{equation}
for all $b\in L^{i}(\R^{n+1},\R^{n}),$ which concludes the proof of the lemma.
\end{proof}
Now, we are going to prove Proposition \ref{pfpsi0} mentioned in the beginning of the sub-section, which we recall here. Before that is important recall that  $$\mathcal{D}_{i}:=\{\nu_{i}: D \to L^{i}_{s}(\R^{n+1},\R^{n}):\nu_{i}(x,0)=0; \mbox{$v_{i}$ is continuous}\}.$$
\begin{prop}\label{pfpsi0.1}
 Under the notation of Definitions \ref{Dis2} and $\ref{psi}$. Let  $1\leq i \leq k$ be a integer and fix a point   $(\overline{\nu}_{0},\ldots,\overline{\nu}_{i-1}) \in \mathcal{D}_{0}\times \mathcal{D}_{1}\times \ldots \times \mathcal{D}_{i-1}.$ Then, the space $D_{i}$ can be endowed with a norm $|.|_{i,D}$ equivalent to the original norm $||.||_{D}$  so that the function $$\Psi^{i}(\overline{\nu}_{0},\ldots,\overline{\nu}_{i-1},\bullet):\mathcal{D}_{i} \to \mathcal{D}_{i}$$  it is a contraction with constant of contraction independent of the point $(\overline{\nu}_{0},\overline{\nu}_{1},\ldots,\overline{\nu}_{i-1}).$
\end{prop}
\begin{proof}[\textbf{Proof of Proposition \ref{pfpsi0.1}}] Let $\nu_{i}\in \mathcal{D}_{i}.$  We define its norm to be
\begin{equation}\label{||i,D}
|\nu_{i}|_{i,D}:=\sup\{|\nu_{i}(x,y)|_{i}:(x,y)\in D\},
\end{equation}
where $|.|_{i}$ is the norm $|.|_{i}$ on $L^{i}(\R^{n+1},\R^{n})$ as in Lemma \ref{|Mib|}. It is easy to check that $|.|_{i,D}$ is a norm on
$\mathcal{D}_{i}$  equivalent to $\norm{.}_{D}.$

 Let $\nu_{i}^{1}=\Psi^{i}(\overline{\nu}_{0},\overline{\nu}_{1},\ldots,\overline{\nu}_{i-1},\mu^{1})$ and  $\nu_{i}^{2}=\Psi^{i}(\overline{\nu}_{0},\overline{\nu}_{1},\ldots,\overline{\nu}_{i-1},\mu^{2}),$ where $ \mu^{1}, \mu^{2}\in \mathcal{D}_{i}$.\\
From Definition \ref{psi} one can deduce that
\begin{eqnarray}\label{psi0.2}
\nu_{i}^{1}-\nu_{i}^{2}&=&\displaystyle (\overline{\nu}_{0}\circ TA-C)i!(1-\overline{\nu}_{0}\circ TB)^{-2}.\mathcal{DC}^{(i,i,i)}((\mu^{1}-\mu^{2}),T)B \nonumber\\
&+&\displaystyle (1-\overline{\nu}_{0}\circ B)^{-1}\mathcal{DC}^{(i,i,i)}((\mu^{1}-\mu^{2}),T)A.
\end{eqnarray}
Recall that, by Eq. \eqref{Dkkk} we have
\begin{equation}\label{Dkkk1}
\mathcal{DC}^{(i,i,i)}((\mu^{1}-\mu^{2}),T)(x,y):=i!\partial_{y}G(x,y) (\mu^{1}-\mu^{2}) \circ T(x,y) \underbrace{\widehat{DT}(x,y)\ldots \widehat{DT}(x,y)}_{k-times},
\end{equation}
where $\widehat{DT}(x,y)$ is as in Eq. \eqref{DT:1}. Hence, in view of \eqref{psi0.2}, \eqref{||i,D} and Lemma \ref{|Mib|} we get
\begin{eqnarray}\label{DT200}
\left|\nu_{i}^{1}-\nu_{i}^{2}\right|_{i,D}&\leq& \left|(\mu_{i}^{1}-\mu_{i}^{2})\right|_{i,D}\frac{(L||A||_{D}+||C||_{D})||B||_{D}}{\norm{\partial_{y}G(x,y)}^{-i}(1-L||B||_{D})^{2}} (i!)^{2}\Lambda(i)\nonumber\\ \nonumber\\
&+&\left|(\mu^{1}-\mu^{2})\right|_{i,D}\frac{||A||_{D}(1-L||B||_{D})}{||\partial_{y}G(x,y)||^{-i}(1-L||B||_{D})^{2}} (i!)^{2}\Lambda(i)\nonumber\\ \nonumber \\
&=&(i!)^{2}\left|(\mu^{1}-\mu^{2})\right|_{i,D} \frac{||A||_{D}+||C||_{D}||B||_{D}}{||\partial_{y}G(x,y)||^{-i}(1-L||B||_{D})^{2}}\Lambda(i),
\end{eqnarray}
where
\begin{equation*}\label{Max:1}
\Lambda(i):=\max_{\stackrel{ m,n \in \N}{ m+n=i}}\{(||A||_{D}+||B||_{D})^{m}(||C||_{D}+1)^{n}\}.
\end{equation*}
But, from Eq. \eqref{L} we have $2||B||_{D}L:=1-||A||_{D}-\sqrt{(1-||A||_{D})^{2}-4||B||_{D}||C||_{D}}.$ Hence,  Eq. \eqref{DT200} becomes
\begin{eqnarray}\label{DT2.0}
\left|\nu_{i}^{1}-\nu_{i}^{2}\right|_{i,D} &\leq& \left|(\mu^{1}-\mu^{2})\right|_{i,D}\Theta(i),
\end{eqnarray}
where
\begin{equation*}
\Theta(i):=\frac{\left(||A||_{D}+||C||_{D}||B||_{D}\right)\displaystyle \max_{m+n=i}\{(||A||_{D}+||B||_{D})^{m}(||C||_{D}+1)^{n}\}}{(2i!)^{-2}||\partial_{y}G||^{-i}\left(1+||A||_{D}+\sqrt{(1-||A||_{D})^{2}-4||B||_{D}||C||_{D}}\right)^{2}}.
\end{equation*}
Moreover, by using Assumption \ref{Assumptions}($L_{3}$) one can see that
\begin{equation}\label{DT2.1}
\Theta(i)<1, \quad 1\leq i \leq k.
\end{equation}
Therefore, on account of Eq. \eqref{DT2.0} one obtains that the function
$$\Psi^{i}(\overline{\nu}_{0},\overline{\nu}_{1},\ldots,\overline{\nu}_{i-1},\bullet):\mathcal{D}_{i} \to \mathcal{D}_{i}$$ is a contraction independent of the point $(\overline{\nu}_{0},\overline{\nu}_{1},\ldots,\overline{\nu}_{i-1}),$ which  finishes the proof.
\end{proof}
Before proceeding to state and prove the following lemma, it is convenient to introduce some useful notation. Consider the following norm-spaces $X_{1},\ldots,X_{n}$ with norm $\norm{.}_{i},$ for $0\leq i\leq k$ respectively and let $X:=X_{1}\times \ldots \times X_{n}.$ Then the norm of the space $X$ will be denoted by $\norm{.}_{_{X}}$ and defined by $\norm{.}_{_{X}}:=\max\{\norm{.}_{i}:1\leq i\leq n\}.$
\begin{lema}\label{l3.0}
Under Definition \ref{psi}. Let $0\leq i\leq k$ be a integer. Suppose that the sets $\mathcal{D}_{j},$ for $1\leq j\leq i$ are endowed with the norm $|.|_{j,D}$ from Proposition \ref{pfpsi0} and the set $X_{i}:=\mathcal{D}_{0}\times \mathcal{D}_{1}\times \ldots \times \mathcal{D}_{i}$ is endowed with the norm $|.|_{X_{i}}.$ Then, the family of maps $\Psi^{(i,\overline{\nu}_{i})}:X_{i-1}\to \mathcal{D}_{i}$ given by $\Psi^{(i,\overline{\nu}_{i})}(\nu_{0},\nu_{1},\ldots, \nu_{i-1})=\Psi^{i}(\nu_{0},\nu_{1},\ldots, \nu_{i-1},\overline{\nu}_{i})$ depends on $\overline{\nu}_{i}$ continuously in the following sense: if $(\nu_{0}^{n},\nu_{1}^{n},\ldots, \nu_{i-1}^{n}) \to (\nu_{0},\nu_{1},\ldots, \nu_{i-1})$ as $n \to \infty$ in the space $X_{i-1},$ then $\Psi^{i}(\nu_{0}^{n},\nu_{1}^{n},\ldots, \nu_{i-1}^{n},\overline{\nu}_{i}) \to \Psi^{i}(\nu_{0},\nu_{1},\ldots, \nu_{i-1},\overline{\nu}_{i})$ in the space $\mathcal{D}_{i}$ for any fixed $\overline{\nu}_{i} \in \mathcal{D}_{i}.$
\end{lema}
\begin{proof}
The proof follows from Definitions \ref{psi}, \ref{U3.N}, \ref{U2.N} and \ref{U1.N}.
\end{proof}
We are going to prove Proposition \ref{pfpsi0}, which we recall here.
\begin{prop}\label{l3.1}
 Assume the notation of Lemma \ref{l3.0}. Then, the function
\begin{equation*}
\widetilde{N}_{i}:X_{i} \to X_{i}
\end{equation*}
 defined by
\begin{equation*}\label{Ni.2}
\widetilde{N}_{i}(\overline{\nu}_{0},\overline{v}_{1},\ldots,\overline{\nu}_{i})=(\Gamma(\overline{\nu}_{0}),\Psi^{1}(\overline{\nu}_{0},\overline{\nu}_{1}),\ldots,\Psi^{i}(\overline{\nu}_{0},\overline{\nu}_{1},\ldots,\overline{\nu}_{i}))
\end{equation*}
has a global attracting fixed point $(A_{0},A_{1},\ldots,A_{i}).$
\end{prop}
\begin{proof}
We proceed by induction on $i$. Suppose that the statement holds for $j$ with $0 \leq j<i$. We wish to show the statement holds for $i.$ To do this; will be proved that the map $\widetilde{N}_{i}=(\widetilde{N}_{i-1},\Psi^{i}):X_{i-1}\times Y \to X\times Y,$ where $X_{i-1}=\mathcal{D}_{0}\times \mathcal{D}_{1}\times \ldots \times \mathcal{D}_{i-1}$ and $Y=\mathcal{D}_{i},$ satisfies the three conditions  conditions of  Fiber Contraction Theorem \ref{FCT}. Indeed,
\begin{itemize}
\item[(a)] By inductive hypothesis the function $\widetilde{N}_{i-1}:X_{i-1}\to X_{i-1}$ has a global attracting fixed point $(A_{0},\ldots,A_{i-1})\in X_{i-1}$.
     \item[(b)] By using  Theorem \ref{pfpsi0} applied to $(A_{0},\ldots,A_{i-1}),$ we have that $$\Psi^{i}(A_{0},\ldots,A_{i-1}, \bullet):\mathcal{D}_{i} \to \mathcal{D}_{i}$$ is a contraction. Then by the  Banach fixed-point theorem  $\Psi^{i}(A_{0},\ldots,A_{i-1},\bullet)$ has an attracting fixed point $A_{i}$.
         \item[(c)]  It follows from Lemma \ref{l3.0} that $\Psi^{i}(.,A_{i}):X \to Y $ is continuous.
    \end{itemize}
   Therefore, from $(a), (b),$ and $ (c),$ we deduce that $\widetilde{N}_{i}:X_{i}\times  X_{i}$ satisfies the three conditions of  Theorem \ref{FCT}. Thus, we conclude that there exists a global attracting fixed point $(A_{0},A_{1},\ldots,A_{i})$ to the function $\widetilde{N}_{i},$ which completes the proof.
\end{proof}
Now we are ready to prove the Proposition \ref{propstep2}, which we recall here.
\begin{prop}
If $\overline{\nu} \in \mathcal{A}_{L}$ is a $C^{k}$ function and  $D^{i}\overline{\nu}(x,0)=0, 0\leq i \leq k$ and $(x,0) \in D_{0}.$ Then the following limit  exists
 \begin{equation*}\label{ME:3}
 \displaystyle \lim_{n\to \infty}(\Gamma^{n}(\overline{\nu}),D(\Gamma^{n}(\overline{\nu})),\ldots,D^{k}(\Gamma^{n}(\overline{\nu})))=(\nu^{*}, A_{1},A_{2},\ldots,A_{k}),
 \end{equation*}
 where $A_{1},A_{2},\ldots,A_{k}$ are continuous functions.
\end{prop}
\begin{proof}[\textbf{Proof of Theorem \ref{propstep2}}]
Let $\overline{\nu} \in \mathcal{A}_{L}$ be a $C^{k}$ function such that  $D^{i}\overline{\nu}(x,0)=0,$ for all $0\leq i \leq k$ and  $(x,0)\in D_{0}.$  By induction, it follows that
\begin{equation*} \label{l4} \widetilde{N}_{i}^{n}(\overline{\nu},D\overline{\nu},\ldots,D^{i}\overline{\nu})=(\Gamma^{n}(\overline{\nu}),D(\Gamma^{n}(\overline{\nu})),\ldots,D^{i}(\Gamma^{n}(\overline{\nu}))).
\end{equation*}
Hence, on account of Proposition \ref{pfpsi0} one obtains
 \begin{equation*}\label{ME:3}
 \displaystyle \lim_{n\to \infty}(\Gamma^{n}(\overline{\nu}),D(\Gamma^{n}(\overline{\nu})),\ldots,D^{k}(\Gamma^{n}(\overline{\nu})))=(\nu^{*}, A_{1},A_{2},\ldots,A_{k}),
 \end{equation*}
 where $A_{j} \in \mathcal{D}_{j},$ for all $1\leq j\leq k,$ which concludes the proof.
\end{proof}
\nocite{Shashkov}
\nocite{Guckenheimer}
\nocite{Guckenheimer1}
\nocite{Williams}
\nocite{Lorenz}
\nocite{Pacifico}
\nocite{Robinson}
\nocite{Shilnikov}
\nocite{Robinson1}
\nocite{Rovella}
\nocite{Afraimovichbook}
\nocite{Belickii}
%\nocite{Deng96}
%\nocite{Deng93}
%\nocite{Deng89}
%\nocite{Deng892}
%\nocite{Sell}
%\nocite{Deng892}
%\nocite{shilnikov68}
%\nocite{shilnikov67}
\def\cprime{$'$} \def\cprime{$'$} \def\cprime{$'$} \def\cprime{$'$}
  \def\cprime{$'$} \def\cprime{$'$} \def\cprime{$'$} \def\cprime{$'$}
  \def\cprime{$'$} \def\cprime{$'$} \def\cprime{$'$}


\begin{thebibliography}{Den89b}

\bibitem[ABS77]{Shilnikov}
V.~S Afra{\u\i}movi{\v{c}}, V.~V. Bykov, and L.~P. Sil{\cprime}nikov.
\newblock The origin and structure of the {L}orenz attractor.
\newblock {\em Dokl. Akad. Nauk SSSR}, 234(2):336--339, 1977.

\bibitem[AP87]{pesin}
V.~S. Afra{\u\i}movich and Ya.~B. Pesin.
\newblock Dimension of {L}orenz type attractors.
\newblock In {\em Mathematical physics reviews, {V}ol.\ 6}, volume~6 of {\em
  Soviet Sci. Rev. Sect. C Math. Phys. Rev.}, pages 169--241. Harwood Academic
  Publ., Chur, 1987.

\bibitem[AP10]{arapa}
V{\'{\i}}tor Ara{\'u}jo and Maria~Jos{\'e} Pacifico.
\newblock {\em Three-dimensional flows}, volume~53 of {\em Ergebnisse der
  Mathematik und ihrer Grenzgebiete. 3. Folge. A Series of Modern Surveys in
  Mathematics [Results in Mathematics and Related Areas. 3rd Series. A Series
  of Modern Surveys in Mathematics]}.
\newblock Springer, Heidelberg, 2010.
\newblock With a foreword by Marcelo Viana.

\bibitem[AV12]{AraVar}
V{\'{\i}}tor Ara{\'u}jo and Paulo Varandas.
\newblock Robust exponential decay of correlations for singular-flows.
\newblock {\em Comm. Math. Phys.}, 311(1):215--246, 2012.

%\bibitem[Bel73]{Belickii}
%G.~R. Belicki{\u\i}.
%\newblock Functional equations, and conjugacy of local diffeomorphisms of
%  finite smoothness class.
%\newblock {\em Funkcional. Anal. i Prilo\v zen.}, 7(4):17--28, 1973.
%
\bibitem[Coh78]{Cohen}
Daniel I.~A. Cohen.
\newblock {\em Basic techniques of combinatorial theory}.
\newblock John Wiley \& Sons, New York-Chichester-Brisbane, 1978.

%\bibitem[Den89a]{Deng89}
%Bo~Deng.
%\newblock Exponential expansion with \v {S}il\cprime nikov's saddle-focus.
%\newblock {\em J. Differential Equations}, 82(1):156--173, 1989.
%
%\bibitem[Den89b]{Deng892}
%Bo~Deng.
%\newblock The \v {S}il\cprime nikov problem, exponential expansion, strong
%  {$\lambda$}-lemma, {$C^1$}-linearization, and homoclinic bifurcation.
%\newblock {\em J. Differential Equations}, 79(2):189--231, 1989.
%
%\bibitem[Den93]{Deng93}
%Bo~Deng.
%\newblock On \v {S}ilnikov's homoclinic-saddle-focus theorem.
%\newblock {\em J. Differential Equations}, 102(2):305--329, 1993.
%
%\bibitem[Den96]{Deng96}
%Bo~Deng.
%\newblock Exponential expansion with principal eigenvalues.
%\newblock {\em Internat. J. Bifur. Chaos Appl. Sci. Engrg.}, 6(6):1161--1167,
%  1996.
%\newblock Nonlinear dynamics, bifurcations and chaotic behavior.

\bibitem[Die69]{Dieudonne}
J.~Dieudonn{\'e}.
\newblock {\em Foundations of modern analysis}.
\newblock Academic Press, New York-London, 1969.
\newblock Enlarged and corrected printing, Pure and Applied Mathematics, Vol.
  10-I.

\bibitem[GH83]{Guckenheimer}
John Guckenheimer and Philip Holmes.
\newblock {\em Nonlinear oscillations, dynamical systems, and bifurcations of
  vector fields}, volume~42 of {\em Applied Mathematical Sciences}.
\newblock Springer-Verlag, New York, 1983.

\bibitem[Guc76]{Guckenheimer1}
J.~Guckenheimer.
\newblock A strange, strange attractor, in the {H}opf bifurcation and its
  applications.
\newblock 19:368--381, 1976.

\bibitem[GW79]{Guckenheimer2}
John Guckenheimer and R.~F. Williams.
\newblock Structural stability of {L}orenz attractors.
\newblock {\em Inst. Hautes \'Etudes Sci. Publ. Math.}, (50):59--72, 1979.

\bibitem[HP69]{HP}
Morris~W. Hirsch and Charles~C. Pugh.
\newblock Stable manifolds for hyperbolic sets.
\newblock {\em Bull. Amer. Math. Soc.}, 75:149--152, 1969.

\bibitem[Jak81]{jakobson}
M.~Jakobson.
\newblock Absolutely continuous invariant measures for one-parameter families
  of one-dimensional maps.
\newblock {\em Comm. Math. Phys.}, 81:39--88, 1981.

\bibitem[Lor63]{Lorenz}
Edward~N. Lorenz.
\newblock Deterministic nonperiodic flow.
\newblock 20:130--141, 1963.

\bibitem[Mey00]{Meyer}
Carl Meyer.
\newblock {\em Matrix analysis and applied linear algebra}.
\newblock Society for Industrial and Applied Mathematics (SIAM), Philadelphia,
  PA, 2000.
\newblock With 1 CD-ROM (Windows, Macintosh and UNIX) and a solutions manual
  (iv+171 pp.).

\bibitem[MPP00]{Pacifico}
C.~A. Morales, M.~J. Pacifico, and E.~R. Pujals.
\newblock Strange attractors across the boundary of hyperbolic systems.
\newblock {\em Comm. Math. Phys.}, 211(3):527--558, 2000.

\bibitem[Rob81]{Robinson}
Clark Robinson.
\newblock Differentiability of the stable foliation for the model {L}orenz
  equations.
\newblock In {\em Dynamical systems and turbulence, {W}arwick 1980 ({C}oventry,
  1979/1980)}, volume 898 of {\em Lecture Notes in Math.}, pages 302--315.
  Springer, Berlin, 1981.

\bibitem[Rob84]{Robinson1}
Clark Robinson.
\newblock Transitivity and invariant measures for the geometric model of the
  {L}orenz equations.
\newblock {\em Ergodic Theory Dynam. Systems}, 4(4):605--611, 1984.

\bibitem[Rov93]{Rovella}
Alvaro Rovella.
\newblock The dynamics of perturbations of the contracting {L}orenz attractor.
\newblock {\em Bol. Soc. Brasil. Mat. (N.S.)}, 24(2):233--259, 1993.

\bibitem[Ryc90]{Rychlik}
Marek~Ryszard Rychlik.
\newblock Lorenz attractors through \v {S}il\cprime nikov-type bifurcation.
  {I}.
\newblock {\em Ergodic Theory Dynam. Systems}, 10(4):793--821, 1990.

%\bibitem[Sel85]{Sell}
%George~R. Sell.
%\newblock Smooth linearization near a fixed point.
%\newblock {\em Amer. J. Math.}, 107(5):1035--1091, 1985.
%
%\bibitem[{\v{S}}il67]{shilnikov67}
%L.~P. {\v{S}}il{\cprime}nikov.
%\newblock On a problem of {P}oincar\'e-{B}irkhoff.
%\newblock {\em Mat. Sb. (N.S.)}, 74 (116):378--397, 1967.
%
%\bibitem[{\v{S}}il68]{shilnikov68}
%L.~P. {\v{S}}il{\cprime}nikov.
%\newblock On the generation of a periodic motion from a trajectory which is
%  doubly asymptotic to a saddle type equilibrium state.
%\newblock {\em Mat. Sb. (N.S.)}, 77 (119):461--472, 1968.

\bibitem[SS94]{Shashkov}
M.~V. Shashkov and L.~P. Shil{\cprime}nikov.
\newblock On the existence of a smooth invariant foliation in {L}orenz-type
  mappings.
\newblock {\em Differential Equations}, 30(4):536--544, 1994.

\bibitem[Via00]{Viana}
Marcelo Viana.
\newblock What's new on {L}orenz strange attractors?
\newblock {\em Math. Intelligencer}, 22(3):6--19, 2000.

\bibitem[Vid14]{vidarte}
Jos\'{e} Vidarte.
\newblock {\em Smooth perturbation of Lorenz-Like flow}.
\newblock Ph.d thesis, ICMC-USP,
  http://www.teses.usp.br/teses/disponiveis/55/55135/tde-15072014-155326/en.php,
  april 2014.

\bibitem[Wil79]{Williams}
R.~F. Williams.
\newblock The structure of {L}orenz attractors.
\newblock {\em Inst. Hautes \'Etudes Sci. Publ. Math.}, (50):73--99, 1979.

\end{thebibliography}
\end{document}